\documentclass[11pt]{article}

\usepackage{amsmath,amssymb,amsthm,amscd,amsfonts}
\usepackage{epsfig,pinlabel,paralist}
\usepackage{mathtools}
\usepackage[vmargin=1in, hmargin=1.25in]{geometry}
\usepackage[font=small,format=plain,labelfont=bf,up,textfont=it,up]{caption}
\usepackage{calc}
\usepackage{etoolbox}
\usepackage{microtype}
\usepackage[all,cmtip]{xy}

\usepackage[bookmarks, bookmarksdepth=2, colorlinks=true, linkcolor=blue, citecolor=blue, urlcolor=blue]{hyperref}

\apptocmd{\thebibliography}{\raggedright}{}{}

\numberwithin{equation}{section}

\theoremstyle{plain}
\newtheorem{theorem}{Theorem}[section]
\newtheorem{maintheorem}{Theorem}
\newtheorem{maincorollary}[maintheorem]{Corollary}
\newtheorem{maintheoremprime}{Theorem}

\newtheorem{lemma}[theorem]{Lemma}

\newtheorem{corollary}[theorem]{Corollary}

\newtheorem{steps}{Step}

\theoremstyle{definition}

\newtheorem{defn}[theorem]{Definition}

\theoremstyle{remark}
\newtheorem{rmk}[theorem]{Remark}
\newenvironment{remark}[1][]{\begin{rmk}[#1] \pushQED{}}{\popQED \end{rmk}}

\newtheorem{eg}[theorem]{Example}

\DeclareMathOperator{\Hom}{Hom}

\DeclareMathOperator{\PMod}{PMod}

\DeclareMathOperator{\Sp}{Sp}

\newcommand\Moduli{\ensuremath{{\mathcal M}}}

\newcommand\C{\ensuremath{\mathbb{C}}}
\newcommand\Z{\ensuremath{\mathbb{Z}}}
\newcommand\Q{\ensuremath{\mathbb{Q}}}

\DeclareMathOperator{\HH}{H}

\newcommand\RH{\ensuremath{\widetilde{\HH}}}

\DeclareMathOperator{\Aut}{Aut}

\DeclareMathOperator{\link}{lk}

\DeclareMathOperator{\Isom}{Isom}

\newcommand\Set[2]{\ensuremath{\left\{\text{#1 $|$
        #2}\right\}}}

\newcommand\cO{\ensuremath{\mathcal{O}}}

\newcommand\cF{\ensuremath{\mathcal{F}}}
\newcommand\cG{\ensuremath{\mathcal{G}}}

\newcommand\cX{\ensuremath{\mathcal{X}}}

\newcommand\tC{\ensuremath{\widetilde{C}}}
\newcommand\tQ{\ensuremath{\widetilde{\Q}}}

\newcommand\tSigma{\ensuremath{\widetilde{\Sigma}}}
\newcommand\tModuli{\ensuremath{\widetilde{\Moduli}}}
\newcommand\tX{\ensuremath{\widetilde{X}}}
\newcommand\tx{\ensuremath{\widetilde{x}}}
\newcommand\tp{\ensuremath{\widetilde{p}}}
\newcommand\tpsi{\ensuremath{\widetilde{\psi}}}

\newcommand\oC{\ensuremath{\overline{C}}}

\newcommand\hMod{\ensuremath{\widehat{\PMod}}}

\newcommand\bbP{\ensuremath{\mathbb{P}}}
\newcommand\bbA{\ensuremath{\mathbb{A}}}
\newcommand\bbH{\ensuremath{\mathbb{H}}}

\newcommand\Curves{\ensuremath{\mathcal{C}}}
\DeclareMathOperator{\vcd}{vcd}
\DeclareMathOperator{\cohcd}{CohCD}
\DeclareMathOperator{\St}{St}
\DeclareMathOperator{\PP}{PP}
\DeclareMathOperator{\NewC}{N\Curves}
\DeclareMathOperator{\Mod}{Mod}

\newcommand{\p}[1]{{\bf #1.}}

\title{\vspace{-40pt}The high-dimensional cohomology of the moduli space of curves with level structures II: punctures and boundary\vspace{-15pt}}
\author{Tara Brendle \and Nathan Broaddus \and Andrew Putman\thanks{Supported in part by NSF grant DMS-1811210}}
\date{}

\begin{document}

\vspace{-10pt}
\maketitle

\vspace{-18pt}
\begin{abstract}
\noindent
We give two proofs that appropriately defined congruence subgroups of the mapping class group of a surface with punctures/boundary
have enormous amounts of rational cohomology in their virtual cohomological dimension.  In particular we give bounds that are super-exponential in each of three variables: number of punctures, number of boundary components, and genus, generalizing work of Fullarton--Putman.  Along the way, we give a simplified account of a theorem of Harer explaining how to relate the homotopy type of the curve complex of a multiply-punctured surface to the curve complex of a once-punctured surface through a process that can be viewed as an analogue of a Birman exact sequence for curve complexes.

As an application, we prove upper and lower bounds on the coherent cohomological dimension of the moduli space of curves with marked points.  For $g \leq 5$, we compute this coherent cohomological dimension for any number of marked points. In contrast to our bounds on cohomology, when the surface has $n \geq1$ marked points, these bounds turn out to be independent of $n$, and depend only on the genus.  
\end{abstract}

\setlength{\parskip}{0pt}
\tableofcontents
\setlength{\parskip}{\baselineskip}

\section{Introduction}
\label{section:introduction}

Let $\Sigma_{g,n}^b$ be an oriented genus $g$ surface with $n$ punctures and $b$ boundary components.  If $n$ or $b$ vanishes, then we will omit them from our notation.  The (pure)
{\em mapping class group} of $\Sigma_{g,n}^b$, denoted $\PMod_{g,n}^b$, is the group
of isotopy classes of orientation-preserving diffeomorphisms of $\Sigma_{g,n}^b$ that
fix $\partial \Sigma_{g,n}^b$ pointwise and do not permute the punctures.\footnote{When $n \in \{0, 1\}$, it is common to use the notation $\Mod$ rather than $\PMod$; however, we will use $\PMod$ consistently throughout the paper for the sake of streamlining certain statements.}
  The cohomology
of $\PMod_{g,n}^b$ plays an important role in many areas of mathematics.  One
fundamental reason for this is that $\PMod_{g,n}^b$ is the orbifold fundamental
group of the moduli space $\Moduli_{g,n}^b$ of finite-volume hyperbolic metrics
on $\Sigma_{g,n}^b$ with geodesic boundary.  The orbifold universal cover of
$\Moduli_{g,n}^b$ is Teichm\"{u}ller space, which is diffeomorphic to a ball and
is thus contractible.  This implies that $\Moduli_{g,n}^b$ is an orbifold
classifying space for $\PMod_{g,n}^b$, and hence that
\[\HH^{\bullet}(\PMod_{g,n}^b;\Q) \cong \HH^{\bullet}(\Moduli_{g,n}^b;\Q).\]
See \cite{FarbMargalitPrimer} for a survey.

\p{Stable cohomology}
There are two important families of cohomology classes for $\PMod_{g,n}^b$:
\begin{compactitem}
\item The Miller--Morita--Mumford classes $\kappa_i \in \HH^{2i}(\PMod_{g,n}^b;\Q)$.
\item For $1 \leq i \leq n$, the Euler class $u_i \in \HH^2(\PMod_{g,n}^b;\Q)$ of
the two-dimensional subbundle of the tangent bundle of $\Moduli_{g,n}^b$ corresponding
to tangent directions where the $i^{\text{th}}$ puncture moves.
\end{compactitem}
A fundamental theorem of Madsen--Weiss \cite{MadsenWeiss} says that the corresponding
map
\[\Q[u_1,\ldots,u_n,\kappa_1,\kappa_2,\ldots] \longrightarrow \HH^{\bullet}(\PMod_{g,n}^b;\Q)\]
of graded rings is an isomorphism in degree $k$ if $g \gg k$ (the reference \cite{MadsenWeiss} only
deals with surfaces without punctures -- see \cite[Proposition 2.1]{LooijengaStable} for how to deal with punctures).  The portion of the
cohomology ring identified by this theorem is the {\em stable cohomology}.  Building
on work of Harer--Zagier \cite{HarerZagier} in the closed case,
Bini--Harer \cite{BiniHarer} proved that the Euler characteristic of
$\Moduli_{g,n}^b$ is enormous (at least if $b=0$), so there must be
a large amount of unstable cohomology.  However, especially in the key closed case 
results on unstable cohomology are only recently beginning to appear.  
Indeed, there are very few known infinite families of nonzero
unstable classes on closed surfaces (found by Payne--Willwacher \cite{PayneWillwacher} and Chan--Galatius--Payne \cite{ChanGalatiusPayne}, disproving
conjectures of Church--Farb--Putman \cite{ChurchFarbPutmanConjecture} and Kontsevich \cite{Kontsevich}).

\p{Top-degree cohomology}
Our main theorem concerns the ``most unstable'' possible cohomology groups.
Harer \cite[Theorem 4.1]{HarerDuality} proved that $\PMod_{g,n}^b$ is a virtual
duality group with virtual cohomological dimension
\begin{equation}
\label{eqn:vcdformula}
\vcd(\PMod_{g,n}^b) =
  \begin{cases*}
   4g - 5 & \text{if $n=b=0$ and $g \geq 2$},\\
   4g + 2b + n - 4 & \text{if $n+b \geq 1$ and $g \geq 1$}.\\
  \end{cases*}
\end{equation}
This identifies the top degree in which $\PMod_{g,n}^b$ might have some
rational cohomology.  However, 
Church--Farb--Putman \cite{ChurchFarbPutmanTop} and
Morita--Sakasai--Suzuki \cite{MoritaSakasaiSuzuki} independently proved that
for $n+b \leq 1$ no rational cohomology can be found in this degree:
\[\HH^{4g-5}(\PMod_g;\Q) = \HH^{4g-3}(\PMod_{g,1};\Q) = \HH^{4g-2}(\PMod_g^1;\Q) = 0 \quad \text{for $g \geq 2$}.\]

\p{Level subgroups without punctures or boundary}
For $\ell \geq 2$, let $\PMod_{g}[\ell]$ be the level-$\ell$
subgroup of $\PMod_{g}$, that is, the kernel of the action of
$\PMod_{g}$ on $\HH_1(\Sigma_{g};\Z/\ell)$.  This action preserves
the algebraic intersection pairing, which is a $\Z/\ell$-valued
symplectic form.  It thus induces a homomorphism
$\PMod_{g} \rightarrow \Sp_{2g}(\Z/\ell)$ that is classically known
to be surjective.  This is all summarized in the short exact sequence
\[1 \longrightarrow \PMod_{g}[\ell] \longrightarrow \PMod_{g} \longrightarrow \Sp_{2g}(\Z/\ell) \longrightarrow 1.\]
Since $\PMod_{g}[\ell]$ is a finite-index subgroup of $\PMod_{g}$,
it is also a virtual duality group (with the same virtual cohomological dimension).  However,
unlike for $\PMod_{g}$ its top rational cohomology group
is not zero.  In fact, a recent theorem of Fullarton--Putman \cite{FullartonPutman} says that
if $p$ is a prime dividing $\ell$, then the dimension of
the rational cohomology of $\PMod_g[\ell]$ in its virtual cohomological dimension 
is at least
\begin{equation}
\label{eqn:fullartonputman}
\frac{1}{g} p^{2g-1} \prod_{k=1}^{g-1} (p^{2k}-1)p^{2k-1}.
\end{equation}
This is super-exponential in $g$; its leading term is $\frac{1}{g} p^{\binom{2g}{2}}$.

\p{Level subgroups with punctures or boundary}
Our main theorem extends this to surfaces with punctures and boundary.
Let $\PMod_{g,n}^b[\ell]$ be the kernel of the action of $\PMod_{g,n}^b$ on
$\HH_1(\Sigma_{g,n}^b;\Z/\ell)$.  The intersection pairing on $\HH_1(\Sigma_{g,n}^b;\Z/\ell)$ is
degenerate if $n+b \geq 2$, so in these cases $\PMod_{g,n}^b$ does not act via the symplectic
group.  

\begin{remark}
It is also common to consider the kernel of the action of $\PMod_{g,n}^b$
on $\HH_1(\Sigma_g;\Z/\ell)$.
This is different from $\PMod_{g,n}^b[\ell]$
when $n+b \geq 2$, and our theorems do not apply to it.
\end{remark}

In recent work, the third author \cite{PutmanStableLevel} has proved that
\[\HH^k(\PMod_{g,n}^b[\ell];\Q) \cong \HH^k(\PMod_{g,n}^b;\Q) \quad \quad \text{when $g \gg k$}.\]
In other words, when you pass to $\PMod_{g,n}^b[\ell]$ no new rational cohomology
appears in the stable range.  When $n+b \leq 1$, this was earlier proved
for $k=1$ by Hain \cite{HainSurvey} and for $k=2$ by the third author \cite{PutmanH2Level}.

\p{Main theorem}
With this definition, our main theorem is as follows.

\begin{maintheorem}
\label{maintheorem:highdim}
Fix some $g \geq 1$ and $n,b \geq 0$ such that $n+b \geq 1$.
Let $\nu$ be the virtual cohomological dimension of $\PMod_{g,n}^b$.  For all $\ell \geq 2$,
the following then hold:
\begin{compactitem}
\item If $n+b = 1$ and $p$ is a prime dividing $\ell$,
then the dimension of $\HH^{\nu}(\PMod_{g,n}^b[\ell];\Q)$ is at least
\[\frac{1}{g} p^{2g-1} \prod_{k=1}^{g-1} (p^{2k}-1)p^{2k-1}.\]
\item If $n+b \geq 2$ and $\nu'$ is the virtual cohomological dimension of
$\PMod_{g,1}$, then the dimension of $\HH^{\nu}(\PMod_{g,n}^b[\ell];\Q)$ is at least
\[\left(\prod_{k=1}^{b+n-1}\left(k \ell^{2g}-1\right)\right) \cdot \dim_{\Q} \HH^{\nu'}(\PMod_{g,1}[\ell];\Q).\]
\end{compactitem}
\end{maintheorem}
Note the latter expression is super-exponential in $n$ and $b$ as well as $g$.
\begin{remark}
Our bound when $n+b = 1$ is the same as Fullarton--Putman's bound in the closed case, 
and in fact follows
easily from their work.  The main point of this paper is to deal
with the case where $n+b \geq 2$.
\end{remark}

\begin{remark}
It is natural to wonder whether the rational cohomology of $\PMod_{g,n}^b$ vanishes in its virtual cohomological
dimension when $n+b \geq 2$ (as was proved by Church--Farb--Putman \cite{ChurchFarbPutmanTop} and
Morita--Sakasai--Suzuki \cite{MoritaSakasaiSuzuki} when $n+b \leq 1$).  This does not hold in all cases; for instance,
calculations of Chan \cite{ChanGenus2} and Bibby--Chan--Gadish--Yun \cite{BibbyChanGadishYun} show that it fails for $\PMod_{2,n}$ for certain $n$, and in fact the rank of this
cohomology group in genus $2$ seem to grow quickly as $n \mapsto \infty$.  However, in unpublished
work we have shown in some cases that vanishing does hold when $n+b$ is small relative to $g$, and we think it is likely that
in its virtual cohomological dimension the rational cohomology of $\PMod_{g,n}^b$ vanishes when $g \gg n+b$.
\end{remark}

\p{Proof I: Birman exact sequence}
We actually give two proofs of Theorem \ref{maintheorem:highdim}.  The first is more direct,
but gives less ancillary information.
It makes use of the Birman exact sequence, which is a short exact sequence
\begin{equation}
\label{eqn:birmanintro}
1 \longrightarrow \pi_1(\Sigma_{g,n-1}^b) \longrightarrow \PMod_{g,n}^b \longrightarrow \PMod_{g,n-1}^b \longrightarrow 1.
\end{equation}
Here the map $\PMod_{g,n}^b \rightarrow \PMod_{g,n-1}^b$ arises from filling in the $n^{\text{th}}$ puncture and the subgroup
$\pi_1(\Sigma_{g,n-1}^b)$ of $\PMod_{g,n}^b$ is the ``point-pushing subgroup'' consisting of mapping classes that drag
the $n^{\text{th}}$ puncture around loops in the surface.  We develop an analogue of \eqref{eqn:birmanintro} for
$\PMod_{g,n}^b[\ell]$ and study the cohomology of the resulting extension.

\p{Steinberg module}
Our second proof uses the Steinberg module for the mapping class group.  Recall from above that Harer \cite{HarerDuality}
proved that $\PMod_{g,n}^b$ is a virtual duality group.  In particular, it is a
$\Q$-duality group.  Letting $\nu$ be its virtual cohomological dimension, this
implies that the cohomology of $\PMod_{g,n}^b$ satisfies the following Poincar\'{e}-duality-like relation:
\[\HH^{\nu-i}(\PMod_{g,n}^b;\Q) \cong \HH_{i}(\PMod_{g,n}^b;\St(\Sigma_{g,n}^b)).\]
The term $\St(\Sigma_{g,n}^b)$ (called the {\em Steinberg module}) is 
the {\em dualizing module} for the mapping class group.  Harer gave a beautiful
description of $\St(\Sigma_{g,n}^b)$ in terms of the curve complex, which we now discuss.

\p{Curve complex}
The curve complex $\Curves_{g,n}^b$ is the simplicial
complex whose $k$-simplices are collections $\{\gamma_0,\ldots,\gamma_k\}$ of isotopy classes of nontrivial simple
closed curves on $\Sigma_{g,n}^b$ that can be realized disjointly.  Here nontrivial means that $\gamma_i$ is neither
nullhomotopic nor homotopic to any of the punctures or boundary components of $\Sigma_{g,n}^b$.  As is
traditional in the subject, we will usually not distinguish between a simple closed curve and its
isotopy class.  Harer proved that
$\Curves_{g,n}^b$ is homotopy equivalent to a wedge of spheres of dimension
\begin{equation}
\label{eqn:lambda}
\lambda = \begin{cases}
2g-2 & \text{if $g \geq 1$ and $n=b=0$},\\
2g-3+n+b & \text{if $g \geq 1$ and $n+b \geq 1$},\\
n+b-4 & \text{if $g = 0$}.
\end{cases}
\end{equation}
The only nonzero reduced homology group of $\Curves_{g,n}^b$ thus lies in degree $\lambda$.  The group $\PMod_{g,n}^b$ acts
on $\Curves_{g,n}^b$, and thus acts on its homology.  
Harer proved that $\St(\Sigma_{g,n}^b) = \RH_{\lambda}(\Curves_{g,n}^b;\Q)$ is the dualizing module of 
$\PMod_{g,n}^b$.

\p{Proof II}
Since groups $\PMod_{g,n}^b[\ell]$ are finite-index subgroups of the virtual duality groups $\PMod_{g,n}^b$, they
are also virtual duality groups with the same dualizing module.  In particular, letting $\nu$ be their
virtual cohomological dimension, we have
\[\HH^{\nu}(\PMod_{g,n}^b[\ell];\Q) \cong \HH_0(\PMod_{g,n}^b[\ell];\St(\Sigma_{g,n}^b)) = (\St(\Sigma_{g,n}^b))_{\PMod_{g,n}^b[\ell]},\]
where the subscript indicates that we are taking coinvariants.  To understand these coinvariants and prove
Theorem \ref{maintheorem:highdim}, we must relate $\St(\Sigma_{g,n}^b)$ to $\St(\Sigma_{g,1})$.  This is the
subject of our next main theorem.

\p{Inductive description of curve complex}
It is clear that $\Curves_{g,n}^b \cong \Curves_{g,n+b}$, so to simplify our notation we will focus on
the surfaces with punctures but no boundary components.  In the spirit of the Birman exact sequence, we can try to relate $\Curves_{g,n}$ and $\Curves_{g,n-1}$.
In \cite{HarerDuality}, Harer proved 
that the map $\Curves_{g,1} \rightarrow \Curves_g$ obtained by removing the puncture
is a homotopy equivalence for all $g \geq 1$.  See
\cite[Proposition 4.7]{HatcherVogtmann} and \cite[Corollary 1.1]{KentLeiningerSchleimer} 
for alternate proofs.  We will thus focus on relating $\Curves_{g,n}$ to $\Curves_{g,n-1}$ for $n \geq 2$.  

In these cases,
the map $\Sigma_{g,n} \rightarrow \Sigma_{g,n-1}$ that fills in the $n^{\text{th}}$ puncture {\em almost} induces
a map $\Curves_{g,n} \rightarrow \Curves_{g,n-1}$.  The only problem is that under this map curves $\gamma$
that bound a twice-punctured disc one of whose punctures is the $n^{\text{th}}$ one become trivial.  Let
$\NewC_{g,n}$ be the set of such curves, and let $\cX_{g,n}$ be the full subcomplex of $\Curves_{g,n}$ spanned by the vertices
that do not lie in $\NewC_{g,n}$.  Any two curves in $\NewC_{g,n}$ must intersect, so none of them are joined
by an edge in $\Curves_{g,n}$.  Regarding $\NewC_{g,n}$ as a discrete set and 
letting $\ast$ denote the simplicial join, we thus have
\[\Curves_{g,n} \subset \NewC_{g,n} \ast \, \cX_{g,n}.\]
Filling in the $n^{\text{th}}$ puncture does induce a map $ \cX_{g,n} \rightarrow \Curves_{g,n-1}$.  This map is
not an isomorphism, but it is implicit in Harer's work that the map $ \cX_{g,n} \rightarrow \Curves_{g,n-1}$
is a homotopy equivalence.  Just like in the case where $n=1$, this can also be
proved as in \cite[Proposition 4.7]{HatcherVogtmann} and \cite[Corollary 1.1]{KentLeiningerSchleimer}.

In fact, even more is true:

\begin{maintheorem}
\label{maintheorem:curvecpx}
Fix some $g \geq 0$ and $n \geq 2$ such that $\Sigma_{g,n} \notin \{\Sigma_{0,2},\Sigma_{0,3}\}$.  Then there is a $\PMod_{g,n}$-equivariant
homotopy equivalence 
$\Curves_{g,n} \simeq \NewC_{g,n} \ast \, \Curves_{g,n-1}$.
\end{maintheorem}

Theorem \ref{maintheorem:curvecpx} was essentially proved by Harer, but we give a self-contained and simplified
account of it using the idea of ``Hatcher flows'' introduced in \cite{HatcherFlow}.  

\p{Inductive description of Steinberg}
Let $\lambda$
be as in \eqref{eqn:lambda}.  Since $\NewC_{g,n}$ is a discrete set, Theorem \ref{maintheorem:curvecpx} implies that
\[\St(\Sigma_{g,n}) = \RH_{\lambda}(\Curves_{g,n};\Q) \cong \RH_{0}(\NewC_{g,n};\Q) \otimes \RH_{\lambda-1}(\Curves_{g,n-1};\Q) = \RH_0(\NewC_{g,n};\Q) \otimes \St(\Sigma_{g,n-1}).\]
This is an isomorphism of $\PMod_{g,n}$-modules, where $\PMod_{g,n}$ acts on $\NewC_{g,n}$ via its action on $\Sigma_{g,n}$ and
on $\St(\Sigma_{g,n-1})$ via the surjection $\PMod_{g,n} \rightarrow \PMod_{g,n-1}$ that fills in the $n^{\text{th}}$ puncture.
For a set $S$, let $\Q[S]$ be the $\Q$-vector space with basis $S$ and let $\tQ[S]$ be the kernel of the augmentation map
$\Q[S] \rightarrow \Q$ taking each element of $S$ to $1$.  The above discussion is summarized in the following corollary.

\begin{maincorollary}
\label{maincorollary:identifysteinberg}
Fix some $g \geq 0$ and $n \geq 2$ such that $\Sigma_{g,n} \notin \{\Sigma_{0,2},\Sigma_{0,3}\}$.  We then have an isomorphism
\[\St(\Sigma_{g,n}) \cong \tQ[\NewC_{g,n}] \otimes \St(\Sigma_{g,n-1})\]
of $\PMod_{g,n}$-modules.
\end{maincorollary}

\p{Alternate proof of corollary}
In addition to the proof of Corollary \ref{maincorollary:identifysteinberg} via
Theorem \ref{maintheorem:curvecpx} described above, we also give an alternate proof
using the Birman exact sequence
\[1 \longrightarrow \pi_1(\Sigma_{g,n-1}) \longrightarrow \PMod_{g,n} \longrightarrow \PMod_{g,n-1} \longrightarrow 1.\]
The free group $\pi_1(\Sigma_{g,n-1})$ is a duality group, and the key 
to our second proof of Corollary \ref{maincorollary:identifysteinberg} is showing
that the dualizing module for $\pi_1(\Sigma_{g,n-1})$ can be identified with
$\tQ[\NewC_{g,n}]$.  This alternate proof of Corollary \ref{maincorollary:identifysteinberg}
is more direct than our proof via Theorem \ref{maintheorem:curvecpx}, but it does
not give the same space-level information that Theorem \ref{maintheorem:curvecpx}
does.  We think that both proofs are enlightening.

\p{Applications to algebraic geometry}
The moduli space $\Moduli_{g,n}$ is a quasi-projective complex variety of dimension
$3g-3+n$.  In \cite{FullartonPutman}, Fullarton--Putman applied their
theorem \eqref{eqn:fullartonputman} to deduce an interesting result about the
algebraic geometry of $\Moduli_g$.  As we describe now, our Theorem \ref{maintheorem:highdim}
allows a generalization of this to $\Moduli_{g,n}$.

For a variety $X$, the {\em coherent cohomological dimension} of $X$, denoted
$\cohcd(X)$, is the maximal dimension $k$ such that there exists a quasi-coherent
sheaf $\mathcal{F}$ on $X$ such that $\HH^k(X;\mathcal{F})\neq 0$.  This measures the geometric
complexity of $X$.  For instance, Serre \cite[Theorem 3.7]{HartshorneBook} proved that $X$ is affine
if and only if $\cohcd(X)=0$.  More generally, if $X$ can be
covered by $m$ open affine subspaces, then Serre's result together with the Mayer--Vietoris
spectral sequence associated to this affine cover implies that $\cohcd(X) \leq m-1$.

Looijenga \cite{Looijenga} conjectured that for $g \geq 2$ the
variety $\Moduli_g$ can be covered by $(g-1)$ open affine subsets, and in particular that
$\cohcd(\Moduli_g) \leq g-2$.  Fullarton--Putman showed how to derive
the opposite inequality $\cohcd(\Moduli_g) \geq g-2$ from \eqref{eqn:fullartonputman}.
In particular $\Moduli_g$ cannot be covered by fewer than $(g-1)$ open affine subsets.
Our generalization of this is as follows:

\begin{maintheorem}
\label{maintheorem:cohcd1}
For $g \geq 2$ and $n \geq 1$, we have $\cohcd(\Moduli_{g,n}) \geq g-1$.
\end{maintheorem}

On first glance, it is striking that this bound is independent of $n$, unlike the bound of Theorem~\ref{maintheorem:highdim}.  
However, using standard techniques from algebraic geometry we can prove the following.

\begin{maintheorem}
\label{maintheorem:cohcd2}
For $g \geq 2$ and $n \geq 1$, we have $\cohcd(\Moduli_{g,n}) \leq \cohcd(\Moduli_g)+1$.
\end{maintheorem}

Combining Theorems \ref{maintheorem:cohcd1} and \ref{maintheorem:cohcd2}, we see that if
$\cohcd(\Moduli_g) = g-2$, then $\cohcd(\Moduli_{g,n}) = g-1$ for all $n \geq 1$.  
For instance, in light of Fullarton--Putman's aforementioned result this
holds if Looijenga's conjecture is true.
Since Fontanari--Pascolutti \cite{FontanariPascolutti}
have proven Looijenga's conjecture for $2 \leq g \leq 5$, we get the following corollary.

\begin{maincorollary}
\label{maincorollary:cohcd3}
 For $g \geq 2$, if $\cohcd(\Moduli_g) = g-2$, then $\cohcd(\Moduli_{g,n}) = g-1$ for all $n \geq 1$. In particular, if $2 \leq g \leq 5$, then equality holds: $\cohcd(\Moduli_{g,n}) = g-1$. 
\end{maincorollary}

\begin{remark}
The derivation of Theorem \ref{maintheorem:cohcd1} from Theorem \ref{maintheorem:highdim} as well as the proof
of Theorem \ref{maintheorem:cohcd2} use fairly standard techniques from algebraic geometry (though there are some
minor complications arising from the fact that these are only coarse moduli spaces).  To make these results
accessible to topologists, we will give a detailed account of how this works.
\end{remark}

\p{Outline}
In \S \ref{section:reductions}, we make some initial reductions.  In \S \ref{section:proofi} we give
our proof of Theorem \ref{maintheorem:highdim} using the Birman exact sequence.  
Next, in \S \ref{section:steinberg} we give two proofs of Corollary 
\ref{maincorollary:identifysteinberg} (which inductively identifies the Steinberg
module), the first by proving Theorem \ref{maintheorem:curvecpx}
and the second using the Birman exact sequence.  This is followed by
\S \ref{section:proofii}, which proves Theorem \ref{maintheorem:highdim} using
our Corollary \ref{maincorollary:identifysteinberg}.  Finally,
in \S \ref{section:algebraicgeometry} we prove Theorems \ref{maintheorem:cohcd1} and \ref{maintheorem:cohcd2}.

\p{Acknowledgments}
We thank Joan Birman for her contributions to an earlier collaboration, funded in part by the Simons Foundation, which helped to lay the groundwork for this paper. We additionally thank MSRI for its support for the first and second authors in this early phase.  We would also like to thank Benson Farb and the University of Chicago for support and funding; much of the work on this paper was completed while the first author was a visitor there.  Finally, we would like
to thank Eric Riedl and Chris Schommer-Pries for helpful conversations.

\section{Initial reductions: removing the boundary and setting up the induction}
\label{section:reductions}

The goal of this section is to reduce Theorem \ref{maintheorem:highdim} to the following
result.

\begin{maintheoremprime}
\label{maintheorem:highdimprime}
Fix some $g \geq 1$ and $n \geq 2$.
Let $\nu$ be the virtual cohomological dimension of $\PMod_{g,n}$.  For
all $\ell \geq 2$, the dimension of $\HH^{\nu}(\PMod_{g,n}[\ell];\Q)$ is at least
\[\left(\left(n-1\right) \ell^{2g}-1\right) \cdot \dim_{\Q} \HH^{\nu-1}(\PMod_{g,n-1}[\ell];\Q).\]
\end{maintheoremprime}

We will give two proofs of Theorem \ref{maintheorem:highdimprime}, one in \S \ref{section:proofi} and the other in \S \ref{section:proofii}.  The rest of this section is devoted
to explaining why it implies Theorem \ref{maintheorem:highdim}.  
The first step is the following lemma, which reduces us to the case of surfaces without
boundary.

\begin{lemma}
\label{lemma:eliminateboundary}
Fix some $g \geq 1$ and $n,b \geq 0$.
Let $\nu$ be the virtual cohomological dimension of $\PMod_{g,n}^b$ and
let $\nu'$ be the virtual cohomological dimension of $\PMod_{g,n+b}$.  Then for all
$\ell \geq 2$ we have
$\HH^{\nu}(\PMod_{g,n}^b[\ell];\Q) \cong \HH^{\nu'}(\PMod_{g,n+b}[\ell];\Q)$.
\end{lemma}
\begin{proof}
There is a surjective map $\rho\colon \PMod_{g,n}^b \rightarrow \PMod_{g,n+b}$ obtained by gluing punctured
discs to each component of $\partial \Sigma_{g,n}^b$ and extending mapping classes by the identity.
The kernel of $\rho$ is the central subgroup $\Z^b$ generated by Dehn twists about the components
of $\partial \Sigma_{g,n}^b$.  This is where we use that $g \geq 1$; indeed, if $g=0$ and $n+b \leq 2$, then there are cases where some of these Dehn twists are either trivial or equal to each
other.  Each of these Dehn twists lies in $\PMod_{g,n}^b[\ell]$, and
$\rho$ restricts to a surjection $\PMod_{g,n}^b[\ell] \rightarrow \PMod_{g,n+b}[\ell]$.  In other words,
we have a central extension
\[1 \longrightarrow \Z^{b} \longrightarrow \PMod_{g,n}^b[\ell] \longrightarrow \PMod_{g,n+b}[\ell] \longrightarrow 1.\]
The associated Hochschild--Serre spectral sequence converging to $\HH^{\bullet}(\PMod_{g,n}^b[\ell];\Q)$ has
\begin{equation}
\label{eqn:eliminatebdry}
E_2^{pq} = \HH^p(\PMod_{g,n+b}[\ell]; \HH^q(\Z^b;\Q)) \cong \HH^p(\PMod_{g,n+b}[\ell];\Q) \otimes \wedge^q \Q^b.
\end{equation}
By Harer's computation of the virtual cohomological dimension of the mapping class group \eqref{eqn:vcdformula}, 
we have $\nu = \nu' + b$.  The only term of \eqref{eqn:eliminatebdry} with $p+q = \nu'+b$ that can possibly be nonzero is
\[E_2^{\nu',b} = \HH^{\nu'}(\PMod_{g,n+b}[\ell];\Q) \otimes \wedge^b \Q^b \cong \HH^{\nu'}(\PMod_{g,n+b}[\ell];\Q).\]
No nonzero differentials can come into or out of this, so it survives to the $E_{\infty}$-page.  We conclude that
\[\HH^{\nu}(\PMod_{g,n}^b[\ell];\Q) = \HH^{\nu'+b}(\PMod_{g,n}^b[\ell];\Q) \cong \HH^{\nu'}(\PMod_{g,n+b}[\ell];\Q). \qedhere\]
\end{proof}

\begin{proof}[Proof of Theorem \ref{maintheorem:highdim}, assuming Theorem \ref{maintheorem:highdimprime}]
We start by recalling the setup.  Let $g \geq 1$ and $n,b \geq 0$ be such
that $n+b \geq 1$.
Let $\nu$ be the virtual cohomological dimension of $\PMod_{g,n}^b$.  Fix some
$\ell \geq 2$.  We must prove the following two things:
\begin{compactitem}
\item If $n+b = 1$ and $p$ is a prime dividing $\ell$,
then the dimension of $\HH^{\nu}(\PMod_{g,n}^b[\ell];\Q)$ is at least 
\begin{equation}
\label{eqn:onetoprove}
\frac{1}{g} p^{2g-1} \prod_{k=1}^{g-1} (p^{2k}-1)p^{2k-1}.
\end{equation}
\item If $n+b \geq 2$ and $\nu'$ is the virtual cohomological dimension of
$\PMod_{g,1}$, then the dimension of $\HH^{\nu}(\PMod_{g,n}^b[\ell];\Q)$ is at least
\[\left(\prod_{k=1}^{b+n-1}\left(k \ell^{2g}-1\right)\right) \cdot \dim_{\Q} \HH^{\nu'}(\PMod_{g,1}[\ell];\Q).\]
\end{compactitem}

For the cases where $n+b = 1$, Harer's computation \eqref{eqn:vcdformula} of the virtual cohomological dimension
of the mapping class group says that it is $4g-3$ for $\PMod_{g,1}$, is $4g-2$ for $\PMod_g^1$, and is
$4g-5$ for $\PMod_g$.  Lemma \ref{lemma:eliminateboundary} along with Bieri--Eckmann duality shows
that
\begin{align*}
\HH^{4g-2}(\PMod_g^1[\ell];\Q) &\cong \HH^{4g-3}(\PMod_{g,1}[\ell];\Q) \\
&\cong \HH_0(\PMod_{g,1}[\ell];\St(\Sigma_{g,1})) = (\St(\Sigma_{g,1}))_{\PMod_{g,1}[\ell]},
\end{align*}
where the subscript indicates that we are taking coinvariants.  Harer \cite{HarerDuality} also proved that the
map $\Curves_{g,1} \rightarrow \Curves_g$ induced by the map that deletes the puncture is a homotopy equivalence.  This
implies that we have a $\PMod_{g,1}$-equivariant isomorphism $\St(\Sigma_{g,1}) \cong \St(\Sigma_g)$, where
$\PMod_{g,1}$ acts on $\St(\Sigma_g)$ via the induced map $\PMod_{g,1} \rightarrow \PMod_g$.  
Continuing the previous calculation and applying Bieri--Eckmann duality again, we see that
\[(\St(\Sigma_{g,1}))_{\PMod_{g,1}[\ell]} \cong (\St(\Sigma_g))_{\PMod_g[\ell]} \cong \HH^{4g-5}(\PMod_g[\ell];\Q).\]
Fullarton--Putman \cite{FullartonPutman} proved that the dimension of this is at least the quantity in \eqref{eqn:onetoprove}.

We now turn to the cases where $n+b \geq 2$.
Let $\nu''$ be the virtual cohomological dimension of $\PMod_{g,n+b}$.  Harer's
computation \eqref{eqn:vcdformula}
of the virtual cohomological dimension of the mapping class group
implies that adding a puncture to a non-closed surface causes the virtual
cohomological dimension to go up by $1$.  It follows that
\[\nu'' = \nu' + (b+n-1).\]
The inequality we must prove now follows from Lemma \ref{lemma:eliminateboundary}
along with repeated applications of Theorem \ref{maintheorem:highdimprime} as follows:
\begin{align*}
\dim_{\Q} \HH^{\nu}(\PMod_{g,n}^b[\ell];\Q) &= \dim_{\Q} \HH^{\nu'+(b+n-1)}(\PMod_{g,n+b}[\ell];\Q) \\
&\geq \left(\left(n+b-1\right) \ell^{2g} - 1\right) \cdot \dim_{\Q} \HH^{\nu'+(b+n-2)}(\PMod_{g,n+b-1}[\ell];\Q) \\
&\geq \left(\left(n+b-1\right) \ell^{2g} - 1\right) \cdot \left(\left(n+b-2\right) \ell^{2g} - 1\right)\\
&\quad\quad\quad\quad\quad\quad \cdot \dim_{\Q} \HH^{\nu'+(b+n-3)}(\PMod_{g,n+b-2}[\ell];\Q) \\
&\geq \cdots \geq \left(\prod_{k=1}^{b+n-1}\left(k \ell^{2g}-1\right)\right) \cdot \dim_{\Q} \HH^{\nu'}(\PMod_{g,1}[\ell];\Q). \qedhere
\end{align*} 
\end{proof}

\section{Proof I: via the Birman exact sequence}
\label{section:proofi}

This section contains our first proof of Theorem \ref{maintheorem:highdimprime}.  The
proof is in \S \ref{section:highbirman}, which is preceded by the preliminary
\S \ref{section:birman} on the Birman exact sequence.

\subsection{The Birman exact sequence}
\label{section:birman}

We start by reviewing the Birman exact sequence \cite{BirmanSeq}.
Fix some $g \geq 0$ and $n \geq 1$ such that $\Sigma_{g,n} \notin \{\Sigma_{1,1}, \Sigma_{0,3}\}$.
There is a surjective map $\PMod_{g,n} \rightarrow \PMod_{g,n-1}$ that fills
in the $n^{\text{th}}$ puncture.  Its kernel, denoted $\PP_{g,n}$, is the
{\em point-pushing subgroup} and satisfies $\PP_{g,n} \cong \pi_1(\Sigma_{g,n-1})$.  This
is where we use our assumption on $\Sigma_{g,n}$; indeed, in the degenerate cases the group
$\PP_{g,n}$ would be trivial.  Informally, elements of $\PP_{g,n}$ ``push'' the
$n^{\text{th}}$ puncture around paths in $\Sigma_{g,n-1}$.  This is all summarized
in the Birman exact sequence
\[1 \longrightarrow \PP_{g,n} \longrightarrow \PMod_{g,n} \longrightarrow \PMod_{g,n-1} \longrightarrow 1.\]
See \cite[\S 4.2]{FarbMargalitPrimer} for a textbook reference.  

We wish to generalize this to $\PMod_{g,n}[\ell]$.  From our definitions, it is clear
that the surjection $\PMod_{g,n} \rightarrow \PMod_{g,n-1}$ restricts to a map
$\rho\colon \PMod_{g,n}[\ell] \rightarrow \PMod_{g,n-1}[\ell]$, though it is not clear that this
restriction is surjective.  Define $\PP_{g,n}[\ell] = \PP_{g,n} \cap \PMod_{g,n}[\ell]$.
We thus have an exact sequence
\[1 \longrightarrow \PP_{g,n}[\ell] \longrightarrow \PMod_{g,n}[\ell] \stackrel{\rho}{\longrightarrow} \PMod_{g,n-1}[\ell].\]
The following theorem provides a Birman exact sequence for level-$\ell$ subgroups; in particular it asserts that $\rho$ is indeed surjective and also gives
a description of $\PP_{g,n}[\ell]$.  

\begin{theorem}
\label{theorem:birmanlevel}
Fix some $g \geq 0$ and $n \geq 1$ such that 
$\Sigma_{g,n} \notin \{\Sigma_{1,1}, \Sigma_{0,3}\}$.  For all $\ell \geq 2$, we have a short exact sequence
\[1 \longrightarrow \PP_{g,n}[\ell] \longrightarrow \PMod_{g,n}[\ell] \stackrel{\rho}{\longrightarrow} \PMod_{g,n-1}[\ell] \longrightarrow 1.\]
Moreover, $\PP_{g,n}[\ell]$ is as follows:
\begin{compactitem}
\item If $n = 1$, then $\PP_{g,n}[\ell] = \PP_{g,n}$.
\item If $n > 1$, then $\PP_{g,n}[\ell]$ is the kernel of the map
\[\PP_{g,n} \cong \pi_1(\Sigma_{g,n-1}) \longrightarrow \HH_1(\Sigma_{g};\Z/\ell)\]
arising from the map $\Sigma_{g,n-1} \rightarrow \Sigma_g$ that fills in the punctures.
\end{compactitem}
\end{theorem}

The third author has proven a similar theorem for the Torelli group
(on a surface with boundary but no punctures) 
\cite[Theorem 1.2]{PutmanCutPaste} and for a variant
of the level-$\ell$ subgroup in some special cases in
\cite[Theorem 2.10]{PutmanH2Level}.  Since our proof is not dramatically
different from these previous proofs, we will omit some routine verifications.

Before we prove Theorem \ref{theorem:birmanlevel}, we need a lemma.

\begin{lemma}
\label{lemma:identifyh1}
Fix some $g \geq 0$ and $n \geq 1$ and $\ell \geq 2$, and let $P$ be a set of $n$ distinct points on $\Sigma_g$.  Identify
$P$ with the punctures of $\Sigma_{g,n}$, so $\PMod_{g,n}$ acts on $\HH_1(\Sigma_g,P;\Z/\ell)$.  Then $\PMod_{g,n}[\ell]$
is the kernel of the action of $\PMod_{g,n}$ on $\HH_1(\Sigma_g,P;\Z/\ell)$.
\end{lemma}
\begin{proof}
By definition, $\PMod_{g,n}[\ell]$ is the kernel of the action on $\HH_1(\Sigma_{g,n};\Z/\ell)$.  Dualizing, this is
the same as the kernel of the action on
$\HH^1(\Sigma_{g,n};\Z/\ell)$.  Let $Q \subset \Sigma_{g,n}$ be a set
consisting of disjoint closed regular neighborhoods of the punctures.  Poincar\'{e}--Lefshetz duality implies that
\[\HH^1(\Sigma_{g,n};\Z/\ell) \cong \HH_1(\Sigma_{g,n},Q;\Z/\ell) \cong \HH_1(\Sigma_g,P;\Z/\ell),\]
where the second isomorphism arising from collapsing each component of $Q$ to a point.  The lemma follows.
\end{proof}

\begin{proof}[Proof of Theorem \ref{theorem:birmanlevel}]
The case $n=1$ follows from the fact that the map $\HH_1(\Sigma_{g,1};\Z/\ell) \rightarrow \HH_1(\Sigma_g;\Z/\ell)$
arising from filling in the puncture is an isomorphism.  We can thus assume that $n \geq 2$.

We first prove that $\rho$ is surjective.
Let $\iota\colon \Sigma_{g,n-2}^1 \hookrightarrow \Sigma_{g,n}$
be an embedding such that $\Sigma_{g,n} \setminus \iota(\Sigma_{g,n-2}^1)$ is an open disc
containing the $n^{\text{th}}$ and $(n-1)^{\text{st}}$ punctures.  We then get
an induced map $\iota_{\ast}\colon \PMod_{g,n-2}^1 \rightarrow \PMod_{g,n}$ that extends
mapping classes over this twice-punctured disc by the identity.  It is immediate from
the definitions that $\iota_{\ast}$ restricts to a map 
$\PMod_{g,n-2}^1[\ell] \rightarrow \PMod_{g,n}[\ell]$.  The composition
\[\PMod_{g,n-2}^1[\ell] \stackrel{\iota_{\ast}}{\longrightarrow} \PMod_{g,n}[\ell] \stackrel{\rho}{\longrightarrow} \PMod_{g,n-1}[\ell]\]
can be identified with the surjective map whose kernel is generated by the Dehn twist
about $\partial \Sigma_{g,n-2}^1$ that arose in Lemma \ref{lemma:eliminateboundary}.
It follows that $\rho$ is surjective.

We now prove that
\begin{equation}
\label{eqn:pptoprove}
\PP_{g,n}[\ell] = \ker\left(\PP_{g,n} \cong \pi_1\left(\Sigma_{g,n-1}\right) \longrightarrow \HH_1\left(\Sigma_{g};\Z/\ell\right)\right).
\end{equation}
Let $P \subset \Sigma_g$ be a set of $n$ distinct points, which we identify with
the punctures of $\Sigma_{g,n}$.  By Lemma \ref{lemma:identifyh1}, the group $\PP_{g,n}[\ell]$ is the kernel
of the action of $\PP_{g,n}$ on $\HH_1(\Sigma_g,P;\Z/\ell)$.  To understand this
action, consider the group
$\hMod_{g,n}[\ell] = \rho^{-1}(\PMod_{g,n-1}[\ell])$, which fits into a short exact sequence
\[1 \longrightarrow \PP_{g,n} \longrightarrow \hMod_{g,n}[\ell] \stackrel{\rho}{\longrightarrow} \PMod_{g,n-1}[\ell] \longrightarrow 1.\]
As we will see, the action of $\hMod_{g,n}[\ell]$ on $\HH_1(\Sigma_g,P;\Z/\ell)$ 
induces a representation whose image is abelian and rather simple.

Let $P' \subset P$ be the first $(n-1)$ punctures.
We have the following isomorphisms: 
\begin{alignat*}{2}
&\HH_1(P,P';\Z/\ell)        &&\cong 0, \\
&\HH_0(P,P';\Z/\ell)        &&\cong \Z/\ell, \\
&\HH_0(\Sigma_g,P';\Z/\ell) &&\cong 0.
\end{alignat*}
The long exact sequence in homology for the triple $(\Sigma_g,P,P')$ thus contains the segment
\begin{equation}
\label{eqn:theexact}
0 \longrightarrow \HH_1(\Sigma_g,P';\Z/\ell) \longrightarrow \HH_1(\Sigma_g,P;\Z/\ell) \stackrel{\partial}{\longrightarrow} \Z/\ell \longrightarrow 0.
\end{equation}
Since $\PMod_{g,n}$ does not permute the points in $P$, its action on $\HH_1(\Sigma_g,P;\Z/\ell)$ preserves this exact sequence.  Its
action on $\HH_1(\Sigma_g,P';\Z/\ell)$ factors through $\rho\colon \PMod_{g,n} \rightarrow \PMod_{g,n-1}$,
and its action on $\Z/\ell$ is trivial.  

It follows that the image of $\hMod_{g,n}[\ell]$ in $\Aut(\HH_1(\Sigma_g,P;\Z/\ell))$ lies
in the subgroup $\Lambda$ consisting of automorphisms 
that preserve the exact sequence \eqref{eqn:theexact} and that act trivially on its kernel
and cokernel.  For $\zeta \in \Hom(\Z/\ell,\HH_1(\Sigma_g,P';\Z/\ell))$, there is a corresponding
automorphism in $\Lambda \subset \Aut(\HH_1(\Sigma_g,P;\Z/\ell))$ defined via the formula
\[x \mapsto x+\zeta(\partial(x)) \quad \quad (x \in \HH_1(\Sigma_g,P;\Z/\ell)).\] 
This correspondence gives rise to an isomorphism, and hence we have
\[\Lambda \cong \Hom(\Z/\ell,\HH_1(\Sigma_g,P';\Z/\ell)) \cong \HH_1(\Sigma_g,P';\Z/\ell).\]

We now return to the subgroup $\PP_{g,n}$ of $\hMod_{g,n}[\ell]$.  Tracing
through the definitions, we see that the restriction of the map
\[\hMod_{g,n}[\ell] \longrightarrow \Lambda \cong \HH_1(\Sigma_g,P';\Z/\ell)\]
to the point-pushing subgroup $\PP_{g,n}$ is precisely the map
\[\PP_{g,n} \longrightarrow \HH_1(\Sigma_g;\Z/\ell) \subset \HH_1(\Sigma_g,P';\Z/\ell)\]
featured in \eqref{eqn:pptoprove}.  The identity \eqref{eqn:pptoprove} follows.
\end{proof}

\subsection{High-dimensional cohomology and the Birman exact sequence}
\label{section:highbirman}

We now use Theorem \ref{theorem:birmanlevel} to prove Theorem \ref{maintheorem:highdimprime}.

\begin{proof}[Proof of Theorem \ref{maintheorem:highdimprime}]
We start by recalling the statement.  Consider some $g \geq 1$ and $n \geq 2$.
Let $\nu$ be the virtual cohomological dimension of $\PMod_{g,n}$.  Finally, fix
some $\ell \geq 2$.  We must prove that $\dim_{\Q} \HH^{\nu}(\PMod_{g,n}[\ell];\Q)$
is at least
\begin{equation}
\label{eqn:birmantoprove}
\left(\left(n-1\right) \ell^{2g}-1\right) \cdot \dim_{\Q} \HH^{\nu-1}(\PMod_{g,n-1}[\ell];\Q).
\end{equation}
We will prove this using the Birman exact sequence from Theorem \ref{theorem:birmanlevel}.

This Birman exact sequence is of the form
\[1 \longrightarrow \PP_{g,n}[\ell] \longrightarrow \PMod_{g,n}[\ell] \longrightarrow \PMod_{g,n-1}[\ell] \longrightarrow 1.\]
The associated Hochschild--Serre spectral sequence converging to $\HH^{\bullet}(\PMod_{g,n}[\ell];\Q)$ has
\begin{equation}
\label{eqn:highbirmanss}
E_2^{pq} = \HH^p(\PMod_{g,n-1}[\ell]; \HH^q(\PP_{g,n}[\ell];\Q)).
\end{equation}
Since 
\[\PP_{g,n}[\ell] \subset \PP_{g,n} \cong \pi_1(\Sigma_{g,n-1})\]
is a free group, its cohomological dimension is $1$.  Also, by Harer's
computation \eqref{eqn:vcdformula} of the virtual cohomological dimension of
the mapping class group, the virtual cohomological dimension of $\PMod_{g,n-1}$ is
$\nu-1$.  It follows that 
the only term of \eqref{eqn:highbirmanss} with $p+q = \nu$ that can possibly be nonzero is
\[E_2^{\nu-1,1} = \HH^{\nu-1}(\PMod_{g,n-1}[\ell];\HH^1(\PP_{g,n}[\ell];\Q)).\]
No nonzero differentials can come into or out of this, so it survives to the $E_{\infty}$-page.  We conclude that
\[\HH^{\nu}(\PMod_{g,n}[\ell];\Q) \cong \HH^{\nu-1}(\PMod_{g,n-1}[\ell];\HH^1(\PP_{g,n}[\ell];\Q)).\]
It is enough, therefore, to prove that $\dim_{\Q} \HH^{\nu-1}(\PMod_{g,n-1}[\ell];\HH^1(\PP_{g,n}[\ell];\Q))$ is at least the quantity \eqref{eqn:birmantoprove}. 

By Theorem \ref{theorem:birmanlevel}, the group $\PP_{g,n}[\ell]$ is the fundamental group of the cover $\tSigma$
of $\Sigma_{g,n-1}$ corresponding to the surjection
\begin{equation}
\label{eqn:coversurjection}
\pi_1(\Sigma_{g,n-1}) \rightarrow \HH_1(\Sigma_{g};\Z/\ell)
\end{equation}
arising from the map $\Sigma_{g,n-1} \rightarrow \Sigma_g$ that fills in the punctures of $\Sigma_{g,n-1}$.  We thus
have $\HH^1(\PP_{g,n}[\ell];\Q) \cong \HH^1(\tSigma;\Q)$.

Since loops around punctures of $\Sigma_{g,n-1}$ lie in the kernel of \eqref{eqn:coversurjection}, the deck group
$\HH_1(\Sigma_g;\Z/\ell)$ acts freely on the punctures of $\tSigma$.  We deduce that $\tSigma$ has
\[(n-1) \cdot |\HH_1(\Sigma_g;\Z/\ell)| = (n-1) \ell^{2g}\]
punctures.  Let $\tSigma'$ be the closed surface obtained by filling in the punctures of $\tSigma$.  We then
have a short exact sequence
\begin{equation}
\label{eqn:punctureseq}
0 \longrightarrow \HH^1(\tSigma';\Q) \longrightarrow \HH^1(\tSigma;\Q) \longrightarrow \Q^{(n-1) \ell^{2g}-1} \longrightarrow 0
\end{equation}
of $\PMod_{g,n-1}[\ell]$-modules whose cokernel arises from the action on the punctures.  Since $\PMod_{g,n-1}[\ell]$ acts
trivially on the deck group $\HH_1(\Sigma_g;\Z/\ell)$, its action on $\Q^{(n-1) \ell^{2g}-1}$ is trivial.

The long exact sequence in $\PMod_{g,n-1}[\ell]$-cohomology associated to the short exact sequence \eqref{eqn:punctureseq} 
contains the segment
\begin{align*}
\HH^{\nu-1}(\PMod_{g,n-1}[\ell];\HH^1(\tSigma;\Q)) \longrightarrow &\HH^{\nu-1}(\PMod_{g,n-1}[\ell];\Q^{(n-1) \ell^{2g}-1}) \\
&\quad\quad\quad\quad\longrightarrow \HH^{\nu}(\PMod_{g,n-1}[\ell];\HH^1(\tSigma';\Q)).
\end{align*}
Since $\nu-1$ is the virtual cohomological dimension of $\PMod_{g,n-1}[\ell]$, we have 
\[\HH^{\nu}(\PMod_{g,n-1}[\ell];\HH^1(\tSigma';\Q)) = 0.\]
We conclude that the central arrow of the composition
\begin{align*}
\HH^{\nu-1}(\PMod_{g,n-1}[\ell];\HH^1(\PP_{g,n}[\ell];\Q)) \cong &\HH^{\nu-1}(\PMod_{g,n-1}[\ell];\HH^1(\tSigma;\Q))\\ 
\rightarrow &\HH^{\nu-1}(\PMod_{g,n-1}[\ell];\Q^{(n-1) \ell^{2g}-1})\\
\cong &\left(\HH^{\nu-1}(\PMod_{g,n-1}[\ell];\Q)\right)^{\oplus ((n-1) \ell^{2g}-1)}.
\end{align*}
is surjective.  It follows that the dimension of $\HH^{\nu-1}(\PMod_{g,n-1}[\ell];\HH^1(\PP_{g,n}[\ell];\Q))$ (and hence also the dimension of
$\HH^{\nu}(\PMod_{g,n}[\ell];\Q)$) is at least the quantity \eqref{eqn:birmantoprove}, as desired.
\end{proof}

\section{Identifying the Steinberg module}
\label{section:steinberg}

This section contains our two proofs of Corollary \ref{maincorollary:identifysteinberg},
which gives an inductive description of the Steinberg module on a punctured surface.
The first is in \S \ref{section:identifycurve}, which proves the stronger
Theorem \ref{maintheorem:curvecpx} which gives an inductive description of the
curve complex.  The second is in \S \ref{section:identifystein}, which works
directly with the Steinberg module.

\subsection{Inductive description of the curve complex}
\label{section:identifycurve}

In this section, we prove Theorem \ref{maintheorem:curvecpx}, which relates
$\Curves_{g,n}$ to $\Curves_{g,n-1}$.  As we described in the introduction, this
implies Corollary \ref{maincorollary:identifysteinberg}.  

We start by proving a technical result.  Stating it will require introducing two pieces of notation.
\begin{compactitem}
\item If $X$ and $Y$ are simplicial complexes, then $X \ast Y$ denotes the join
of $X$ and $Y$.  By definition, this is the simplicial complex whose vertices
are $X^{(0)} \sqcup Y^{(0)}$ and whose simplices correspond to sets $\sigma$ of vertices
such that $\sigma \cap X^{(0)}$ (resp.\ $\sigma \cap Y^{(0)}$) is either $\emptyset$
or a simplex of $X$ (resp.\ $Y$).
\item If $Z$ is a simplicial complex and $v \in Z^{(0)}$ is a vertex of $Z$, then
$\link_Z(v)$ is the subcomplex of $Z$ whose simplices correspond to sets $\sigma$ of vertices
such that $\sigma \cup \{v\}$ is a simplex of $Z$ but such that $v \notin \sigma$.
\end{compactitem}
Our technical result is then as follows.

\begin{lemma}
\label{lemma:identifyjoin}
Let $X$ be a simplicial complex, let $I$ be a discrete set, and let $Y$ be a subcomplex
of the join $I \ast X$ such that $I, X \subset Y$.  Assume that for all $i \in I$, the
inclusion $\link_Y(i) \hookrightarrow X$ is a homotopy equivalence.  Then 
the inclusion $Y \hookrightarrow I \ast X$ is a homotopy equivalence.
\end{lemma}
\begin{proof}
It is enough to find open covers 
$\{U_{j}\}_{j \in J}$ of $Y$
and $\{V_{j}\}_{j \in J}$ of $I \ast X$ with the following two properties  \cite[Corollary 4K.2]{HatcherBook} :
\begin{enumerate}
\item For all $j \in J$, we have $U_j \subset V_j$.
\item For all $k \geq 1$ and all distinct $j_1,\ldots,j_k \in J$, the inclusion 
$U_{j_1} \cap \cdots \cap U_{j_k} \hookrightarrow V_{j_1} \cap \cdots \cap V_{j_k}$
is a homotopy equivalence
\end{enumerate}
For $i \in I$, let $X_i = \link_Y(i)$.
Set $J = I \sqcup \{0\}$.  We define the covers $\{ U_j \}$ of $Y$ and $\{V_j \}$ of $ I \ast X$ as follows.  To begin, we set $U_0 = Y \setminus I$ and $V_0 = (I \ast X) \setminus I$.
Next, for $j \in I$, we set $U_j = (\{j\} \ast X_j) \setminus X_j$ and $V_j = (\{j\} \ast X) \setminus X$.
We emphasize in these definitions that the $U_j$ and $V_j$ are open sets and not
subcomplexes.  

It follows immediately from the definitions that $U_j \subset V_j$ for all $j \in J$, so it remains to check the second condition above.  First, we note that for all distinct $i_1, i_2 \in I$, we have $U_{i_1} \cap U_{i_2} = \emptyset$ and $V_{i_1} \cap V_{i_2} = \emptyset$.  There are thus just three cases to check.  
\begin{itemize}
\item[{\bf Case 1:}] The inclusion $U_0 \hookrightarrow V_0$ is a homotopy equivalence since both $U_0$ and $V_0$ deformation
retract to $X$.
\item[{\bf Case 2:}] For $j \in I$, the inclusion $U_j \hookrightarrow V_j$ is a homotopy equivalence since both $U_j$ and
$V_j$ are contractible.
\item[{\bf Case 3:}] For $j \in I$, the inclusion $U_0 \cap U_j \hookrightarrow V_0 \cap V_j$ can
be seen to be a homotopy equivalence as follows.  By construction,
$U_0 \cap U_j \cong X_j \times (0,1)$ and
$V_0 \cap V_j \cong X \times (0,1)$.  Under these identifications, the
inclusion $U_0 \cap U_j \hookrightarrow V_0 \cap V_j$ is identified with the
inclusion $X_j \times (0,1) \hookrightarrow X \times (0,1)$ induced by inclusion
$X_j \hookrightarrow X$.  This is a homotopy equivalence by assumption.\qedhere
\end{itemize}
\end{proof}

\begin{proof}[Proof of Theorem \ref{maintheorem:curvecpx}]
Let us first recall what we must prove.  Fix some $g \geq 0$ and $n \geq 2$ such 
that $\Sigma_{g,n} \notin \{\Sigma_{0,2},\Sigma_{0,3}\}$.
Recall that $\NewC_{g,n}$ is the set of vertices of $\Curves_{g,n}$ consisting
of simple closed curves $\gamma$ on $\Sigma_{g,n}$ that bound a twice-punctured
disc one of whose punctures is the $n^{\text{th}}$ one.  We remark that
our assumption $\Sigma_{g,n} \notin \{\Sigma_{0,2},\Sigma_{0,3}\}$ ensures
that such curves are nontrivial.  Our goal is to
construct a $\PMod_{g,n}$-equivariant homotopy equivalence
$\Curves_{g,n} \simeq \NewC_{g,n} \ast \, \Curves_{g,n-1}$.

We start with the following observation.  Consider some $\gamma \in \NewC_{g,n}$.
Let $D \subset \Sigma_{g,n}$ be the closed disc bounded by $\gamma$, so $D$ contains
two punctures.  Since
$\link_{\Curves_{g,n}}(\gamma)$ is the full subcomplex of $\Curves_{g,n}$ spanned
by nontrivial curves in $\Sigma_{g,n} \setminus D \cong \Sigma_{g,n-1}$, we have
that $\link_{\Curves_{g,n}}(\gamma) \cong \Curves_{g,n-1}$.

Define $\cX_{g,n}$ to be the full subcomplex of $\Curves_{g,n}$ spanned by the
vertices that do {\em not} lie in $\NewC_{g,n}$.  
Since any two curves in $\NewC_{g,n}$ must intersect and thus are not joined by an edge in
$\Curves_{g,n}$, we have
\begin{equation}
\label{eqn:includecurves}
\Curves_{g,n} \subset \NewC_{g,n} \ast \, \cX_{g,n}.
\end{equation}
Below we will prove that for all $\gamma \in \NewC_{g,n}$, the
inclusion $\link_{\Curves_{g,n}}(\gamma) \hookrightarrow \cX_{g,n}$ is a homotopy equivalence.
Having done this, we can apply Lemma \ref{lemma:identifyjoin} and conclude that 
the inclusion \eqref{eqn:includecurves} is a homotopy equivalence.  

From this, the theorem can be deduced as follows.  Deleting the $n^{\text{th}}$
puncture defines a map $\pi\colon \cX_{g,n} \rightarrow \Curves_{g,n-1}$.  Choose some
$\gamma \in \NewC_{g,n}$.  Since the inclusion
$\link_{\Curves_{g,n}}(\gamma) \hookrightarrow \cX_{g,n}$ is a homotopy equivalence and the composition
\begin{equation}
\label{eqn:linkiso}
\link_{\Curves_{g,n}}(\gamma) \hookrightarrow \cX_{g,n} \stackrel{\pi}{\rightarrow} \Curves_{g,n-1}
\end{equation}
is an isomorphism, the map $\pi\colon \cX_{g,n} \rightarrow \Curves_{g,n-1}$ is a homotopy
equivalence.  We thus have a sequence of homotopy equivalences
\[\Curves_{g,n} \simeq \NewC_{g,n} \ast \, \cX_{g,n} \simeq \NewC_{g,n} \ast \, \Curves_{g,n-1},\]
as desired.

Fixing some $\gamma \in \NewC_{g,n}$, it remains to prove that the inclusion
$\iota\colon \link_{\Curves_{g,n}}(\gamma) \hookrightarrow \cX_{g,n}$ is a homotopy equivalence.  Since
the composition \eqref{eqn:linkiso} is an isomorphism, the map $\iota$ induces an injection
on homotopy groups.  It is thus enough to prove that it induces a surjection as well.  Fixing
some $m \geq 0$ and some simplicial complex structure on $S^m$, 
let $\psi\colon S^m \rightarrow \cX_{g,n}$ be a simplicial map.  Our goal is to homotope
$\psi$ such that its image lies in $\link_{\Curves_{g,n}}(\gamma)$.  

Let $\alpha$ be an arc in the disc bounded by $\gamma$ that connects the two punctures
in that disc:\\
\centerline{\psfig{file=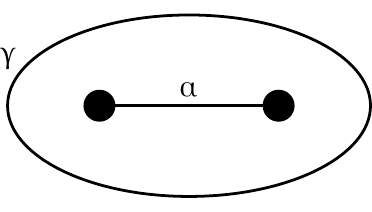,scale=100}}
The curve $\gamma$ is thus the boundary of a regular neighborhood of $\alpha$, and a simple
closed curve on $\Sigma_{g,n}$ is isotopic to a curve that is disjoint from $\gamma$
if and only if it is isotopic to a curve that is disjoint from $\alpha$.  It
follows that it is enough to homotope $\psi$ such that for all vertices $v$ of $S^m$,
the curve $\psi(v)$ can be isotoped so as to be disjoint from $\alpha$.  We
will do this using the ``Hatcher flow'' idea introduced in \cite{HatcherFlow}.

Let $v_1,\ldots,v_r$ be the vertices of $S^m$.  For each $1 \leq i \leq r$, pick
a simple closed curve representative $\delta_i$ of the isotopy class $\psi(v_i)$ with
the following key property:
\begin{compactitem}
\item For all distinct $1 \leq i,j \leq r$, if $\psi(v_i)$ and $\phi(v_j)$ are isotopic
to disjoint simple closed curves, then $\delta_i$ and $\delta_j$ are disjoint.  This holds
in particular if $v_i$ and $v_j$ are adjacent in $S^m$.
\end{compactitem}
For instance, fixing a hyperbolic metric on $\Sigma_{g,n}$, we can let $\delta_i$
be the geodesic in the isotopy class $\psi(v_i)$ and then if $\delta_i = \delta_j$
for some $i \neq j$ perturb them slightly so as to be disjoint.  Perturbing
the $\delta_i$ slightly, we can assume that they intersect $\alpha$ transversely
and that no two of the $\delta_i$ intersect $\alpha$ in the same point.

Let $p_1,\ldots,p_s$ be the intersection points of the $\delta_i$ with $\alpha$, enumerated
in their natural order starting from the one closest to the endpoint of
$\alpha$ at the $n^{\text{th}}$ puncture.  For $1 \leq j \leq s$, let $\delta_{i_j}$ be the simple closed curve
containing $p_j$.  As in the following figure, we can ``slide'' these intersection 
points off of the arc $\alpha$ one at a time by pulling them over the $n^{\text{th}}$ puncture:\\
\centerline{\psfig{file=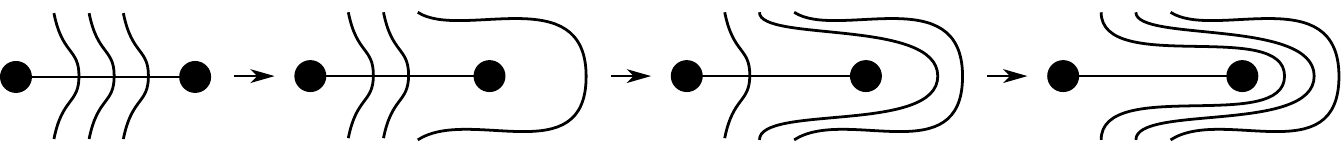,scale=100}}
When $p_j$ is slid off, the curve $\delta_{i_j}$ in $\cX_{g,n}$ is replaced with a new curve $\delta'_{i_j}$.  
The new curve $\delta'_{i_j}$ also lies in $\cX_{g,n}$.  Indeed, 
being in $\cX_{g,n}$ is equivalent to remaining nontrivial
after deleting the $n^{\text{th}}$ puncture, and after
deleting this puncture $\delta_{i_j}$ and $\delta'_{i_j}$ become homotopic.

Set $\psi_0 = \psi$, and for $1 \leq j \leq s$, let $\psi_j\colon S^m \rightarrow \cX_{g,n}$
be the simplicial map obtained from $\psi$ by sliding off $p_1,\ldots,p_j$ as
above (and thus changing $\delta_{i_1},\ldots,\delta_{i_j}$; of course, some
of the $\delta_i$ will be changed multiple times in this process).  

It is clear
from our construction that $\psi_j\colon S^m \rightarrow \cX_{g,n}$ is homotopic to $\psi_{j+1}\colon S^m \rightarrow \cX_{g,n}$ 
for $0 \leq j < s$, so 
$\psi = \psi_0$ is homotopic to $\psi_s$.  Since each simple closed
curve in the image of $\psi_s$ is disjoint from $\alpha$, the image of $\psi_s$
lies in $\link_{\Curves_{g,n}}(\gamma)$, as desired.
\end{proof}

\subsection{Directly identifying the Steinberg module}
\label{section:identifystein}

In this section, we give a direct proof of Corollary \ref{maincorollary:identifysteinberg}.
Recall that this asserts that for $g \geq 0$ and $n \geq 2$ such that 
$\Sigma_{g,n} \notin \{\Sigma_{0,2},\Sigma_{0,3}\}$, we have
\[\St(\Sigma_{g,n}) \cong \tQ[\NewC_{g,n}] \otimes \St(\Sigma_{g,n-1}).\]
Here $\St(\Sigma_{g,n})$ and $\St(\Sigma_{g,n-1})$ are the dualizing modules for 
the $\Q$-duality groups $\PMod_{g,n}$ and $\PMod_{g,n-1}$,
respectively.  We will deduce this from the Birman exact sequence
\begin{equation}
\label{eqn:steinbirman}
1 \longrightarrow \pi_1(\Sigma_{g,n-1}) \longrightarrow \PMod_{g,n} \longrightarrow \PMod_{g,n-1} \longrightarrow 1
\end{equation}
discussed in \S \ref{section:birman}.

\p{Reduction}
The free group $\pi_1(\Sigma_{g,n-1})$ is a $\Q$-duality group (in fact, even
a $\Z$-duality group) of dimension $1$.
Letting $D$ be
the dualizing module for the $\Q$-duality group $\pi_1(\Sigma_{g,n-1})$, we can apply a basic
theorem of Bieri--Eckmann (see \cite[Theorem 9.10]{BieriBook}) about extensions
of $\Q$-duality groups to \eqref{eqn:steinbirman} and deduce that the dualizing module
$\St(\Sigma_{g,n})$ of $\PMod_{g,n}$ is isomorphic to
$D \otimes \St(\Sigma_{g,n-1})$.  Here $\PMod_{g,n}$ acts on $\St(\Sigma_{g,n-1})$
via the surjection $\PMod_{g,n} \rightarrow \PMod_{g,n-1}$ and on $D$
via the conjugation action of $\PMod_{g,n}$ on its normal subgroup
$\pi_1(\Sigma_{g,n-1})$.  It follows that to prove Corollary 
\ref{maincorollary:identifysteinberg}, it is enough to prove the following key lemma.

\begin{lemma}
\label{lemma:dualizingfree}
Fix some $g \geq 0$ and $n \geq 2$ such that 
$\Sigma_{g,n} \notin \{\Sigma_{0,2},\Sigma_{0,3}\}$.  There is then 
a $\PMod_{g,n}$-equivariant isomorphism between the dualizing module
of $\pi_1(\Sigma_{g,n-1})$ and $\tQ[\NewC_{g,n}]$.
\end{lemma}

\p{Dualizing modules}
Before we prove Lemma \ref{lemma:dualizingfree}, 
we need to discuss some general facts about duality groups for
which \cite[Chapter 3]{BieriBook} and \cite[Chapter VIII]{BrownCohomology} form
appropriate textbook references.  Fix a group $G$.
The group $G$ is a $\Q$-duality group of dimension $n$ if and only if $\HH^k(G;\Q[G]) = 0$ for all $k \neq n$.  The dualizing
module for $G$ is then $\HH^n(G;\Q[G])$, on which $G$ acts via its right action on $\Q[G]$.

\p{Topological interpretation}
We now give a topological description of $\HH^n(G;\Q[G])$.  Assume that $(X,x_0)$ is a 
based compact simplicial complex that forms an Eilenberg--MacLane space for $G$, so
$\pi_1(X,x_0) = G$ and the based universal cover $(\tX,\tx_0)$ is contractible.
The $\pi_1(X,x_0)$-module $\Q[G]$ is a local system on $X$, and
$\HH^n(G;\Q[G]) \cong \HH^n(X;\Q[G])$.  Since $X$ is a compact simplicial complex,
there is a natural isomorphism between $\HH^n(X;\Q[G])$ and the compactly supported
cohomology $\HH^n_c(\tX;\Q)$ of $\tX$.

\p{Action of automorphisms}
The cohomology of a group with twisted coefficients forms a bifunctor which
is contravariant in the group and covariant in the coefficient system.  More precisely,
if for $i=1,2$ we have groups $G_i$ equipped with coefficient systems $M_i$, then
given a group homomorphism $f\colon G_2 \rightarrow G_1$ and a morphism
$f'\colon M_1 \rightarrow M_2$ such that
\[f'(f(\gamma) \cdot m) = \gamma \cdot f'(m) \quad \quad (\gamma \in G_2, m \in M_1),\]
we get induced maps 
\[(f,f')_{\ast}\colon \HH^n(G_1;M_1) \rightarrow \HH^n(G_2;M_2)\]
for all $n$.  In particular, the group $\Aut(G)$ acts on $\HH^n(G;\Q[G])$ via
the formula
\[\phi_{\ast}(x) = (\phi^{-1},\phi)_{\ast}(x) \quad \quad (\phi \in \Aut(G), x \in \HH^n(G;\Q[G]).\]

\p{Topological interpretation of action}
Again let $(X,x_0)$ be a based compact simplicial complex that forms an Eilenberg--MacLane
space for $G$ and let $(\tX,\tx_0)$ be the based universal cover of $(X,x_0)$.  The
group $\Aut(G)$ is isomorphic to the group of homotopy classes of basepoint-preserving
self-homotopy-equivalences of $(X,x_0)$.  Consider $\phi \in \Aut(G)$, and
let $\psi\colon (X,x_0) \rightarrow (X,x_0)$ be a homotopy equivalence realizing
$\phi^{-1}$.  We can then lift $\psi$ to a proper map 
$\tpsi\colon (\tX,\tx_0) \rightarrow (\tX,\tx_0)$.  The map
\[\phi_{\ast} = (\phi^{-1},\phi)_{\ast}\colon \HH^n(G;\Q[G]) \rightarrow \HH^n(G;\Q[G])\] 
described above
can then be identified with 
$\tpsi^{\ast} \colon \HH^n_c(\tX;\Q) \rightarrow \HH^n_c(\tX;\Q)$ in the sense
that the diagram
\[\begin{CD}
\HH^n_c(\tX;\Q) @>{\tpsi^{\ast}}>> \HH^n_c(\tX;\Q) \\
@V{\cong}VV                        @V{\cong}VV \\
\HH^n(G;\Q[G])  @>{(\phi^{-1},\phi)_{\ast}}>>  \HH^n(G;\Q[G])
\end{CD}\]
commutes.

\p{The proof}
We can now prove Lemma \ref{lemma:dualizingfree}.

\begin{proof}[Proof of Lemma \ref{lemma:dualizingfree}]
We first recall the statement, using the notation we introduced above.  
Fix some $g \geq 0$ and $n \geq 2$ such that 
$\Sigma_{g,n} \notin \{\Sigma_{0,2},\Sigma_{0,3}\}$.  We must prove that there is an isomorphism
\[\HH^1(\pi_1(\Sigma_{g,n-1});\Q[\pi_1(\Sigma_{g,n-1})]) \cong \tQ[\NewC_{g,n}]\]
of $\PMod_{g,n}$-modules.  We will do this in two steps:

\begin{steps}
We construct an isomorphism
\[\eta\colon \HH^1(\pi_1(\Sigma_{g,n-1});\Q[\pi_1(\Sigma_{g,n-1})]) \stackrel{\cong}{\longrightarrow} \tQ[\NewC_{g,n}]\]
of $\pi_1(\Sigma_{g,n-1})$-modules.
\end{steps}

Let $p_n \in \Sigma_{g,n-1}$ be the basepoint for $\pi_1$.  The group $\PMod_{g,n}$
acts on $\pi_1(\Sigma_{g,n-1},p_n)$ by identifying $p_n$ with the $n^{\text{th}}$
puncture of $\Sigma_{g,n}$.  Our assumptions on $g$ and $n$ imply that
$\Sigma_{g,n-1}$ can be endowed with a hyperbolic metric of finite volume.  
Fixing such a metric identifies the based universal cover of
$(\Sigma_{g,n-1},p_n)$ with $\rho\colon(\bbH^2,\tp_n) \rightarrow (\Sigma_{g,n-1},p_n)$
for some basepoint $\tp_n \in \bbH^2$.

Let $C_1,\ldots,C_{n-1} \subset \Sigma_{g,n-1}$ be open neighborhoods of the cusps
of $\Sigma_{g,n-1}$ satisfying the following properties:
\begin{compactitem}
\item For all $1 \leq i \leq n-1$, the preimage $\tC_i \subset \bbH^2$ of $C_i$
is a disjoint union of open horoballs.
\item For distinct $1 \leq i,j \leq n-1$, the closures $\oC_i$ and $\oC_j$ of
$C_i$ and $C_j$ in $\Sigma_{g,n-1}$ are disjoint.
\item For all $1 \leq i \leq n-1$, we have $p_n \notin \oC_i$.
\end{compactitem}
Define
\[X = \Sigma_{g,n-1} \setminus \bigcup_{i=1}^{n-1} C_i \quad \text{and} \quad
\tX = \bbH^2 \setminus \bigcup_{i=1}^{n-1} \tC_i,\]
so $X$ is a compact aspherical $2$-dimensional manifold with boundary such that
$\pi_1(X,p_n) = \pi_1(\Sigma_{g,n-1},p_n)$ and $\rho\colon (\tX,\tp_n) \rightarrow (X,p_n)$
is its based universal cover.

As was discussed before the proof, since $X$ is a compact Eilenberg--MacLane
space for $\pi_1(\Sigma_{g,n-1},p_n)$ we have an isomorphism
\begin{equation}
\label{eqn:firstiso}
\HH^1(\pi_1(\Sigma_{g,n-1},p_n);\Q[\pi_1(\Sigma_{g,n-1},p_n)]) \cong \HH^1_c(\tX;\Q)
\end{equation}
of $\pi_1(\Sigma_{g,n-1},p_n)$-modules.  By Poincar\'{e}-Lefschetz duality 
and the long exact sequence of the pair $(\tX,\partial \tX)$, we have
\begin{equation}
\label{eqn:secondiso}
\HH^1_c(\tX;\Q) \cong \HH_1(\tX,\partial \tX;\Q) \cong \RH_0(\partial \tX;\Q).
\end{equation}
Let $B$ be the set of components of $\partial \tX$, so
\begin{equation}
\label{eqn:thirdiso}
\RH_0(\partial \tX;\Q) \cong \tQ[B].
\end{equation}
The elements of $B$ are in bijection with the connected
components of $\tC_1 \cup \cdots \cup \tC_{n-1}$, which themselves are
in bijection with the maximal parabolic subgroups of
$\pi_1(\Sigma_{g,n-1},p_n) \subset \Isom(\bbH^2)$.  
These are precisely the subgroups generated by elements of
$\pi_1(\Sigma_{g,n-1},p_n)$ that are freely homotopic to loops 
surrounding one of the punctures.  

These loops are simple, and
each maximal parabolic subgroup
is generated by two such loops, one with the cusp on its left and the
other with the cusp on its right.  It follows that $B$ is in bijection
with homotopy classes of $p_n$-based {\em unoriented} loops on $\Sigma_{g,n-1}$,
and as the following picture shows, these are in bijection with elements
of $\NewC_{g,n}$ (which recall is the set of homotopy classes of simple closed
curves on $\Sigma_{g,n}$ bounding twice-punctured discs one of whose punctures
is the $n^{\text{th}}$ one, which we identify with $p_n$):\\
\centerline{\psfig{file=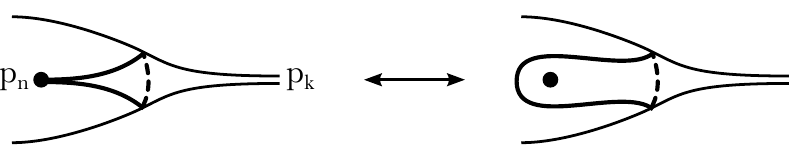,scale=100}}
In this figure, the $k^{\text{th}}$ puncture is suggestively labeled $p_k$, and the loop
on the right bounds a disc containing the $k^{\text{th}}$ and $n^{\text{th}}$ punctures
(as in the definition of $\NewC_{g,n}$).
We conclude from this discussion that we have an isomorphism
\begin{equation}
\label{eqn:fourthiso}
\tQ[B] \cong \tQ[\NewC_{g,n}]
\end{equation}
of $\pi_1(\Sigma_{g,n-1},p_n)$-modules.

Combining \eqref{eqn:firstiso}, \eqref{eqn:secondiso}, \eqref{eqn:thirdiso}, and
\eqref{eqn:fourthiso}, we obtain an isomorphism
\[\eta\colon \HH^1(\pi_1(\Sigma_{g,n-1});\Q[\pi_1(\Sigma_{g,n-1})]) \stackrel{\cong}{\longrightarrow} \tQ[\NewC_{g,n}]\]
of $\pi_1(\Sigma_{g,n-1},p_n)$-modules.  

\begin{steps}
We prove that the map $\eta$ is an isomorphism of $\PMod_{g,n}$-modules.
\end{steps}

We will continue using the notation defined in the previous step.
Consider some $\phi \in \PMod_{g,n}$.  Our goal is to prove that $\eta$ identifies
the action of $\phi$ on $\HH^1(\pi_1(\Sigma_{g,n-1});\Q[\pi_1(\Sigma_{g,n-1})])$
induced by the conjugation action of $\PMod_{g,n}$ on its normal subgroup
$\pi_1(\Sigma_{g,n-1})$ with the evident action of $\phi$ on $\tQ[\NewC_{g,n}]$.

As was discussed before the proof, let 
$\psi\colon (X,p_n) \rightarrow (X,p_n)$ be an orientation-preserving
diffeomorphism representing the restriction of the inverse of $\phi$ to
$X \subset \Sigma_{g,n-1}$.  Let $\tpsi\colon (\tX,\tp_n) \rightarrow (\tX,\tp_n)$
be the lift of $\tpsi$.  We then have a commutative diagram
\[\begin{CD}
\HH_1(\tX,\partial \tX;\Q) @<{\tpsi_{\ast}}<< \HH_1(\tX,\partial \tX;\Q) \\
@V{\cong}VV                        @V{\cong}VV \\
\HH^n_c(\tX;\Q) @>{\tpsi^{\ast}}>> \HH^n_c(\tX;\Q) \\
@V{\cong}VV                        @V{\cong}VV \\
\HH^n(\pi_1(\Sigma_{g,n-1});\Q[\pi_1(\Sigma_{g,n-1})])  @>{(\phi^{-1},\phi)_{\ast}}>>  \HH^n(\pi_1(\Sigma_{g,n-1});\Q[\pi_1(\Sigma_{g,n-1})]).
\end{CD}\]
In other words, under our isomorphisms the action of $\tpsi_{\ast}^{-1}$ on
$\HH_1(\tX,\partial \tX;\Q)$.
is identified with the action of $\phi$ on
$\HH^n(\pi_1(\Sigma_{g,n-1});\Q[\pi_1(\Sigma_{g,n-1})])$.  
Under the isomorphism
\[\HH_1(\tX,\partial \tX;\Q) \cong \RH_0(\partial \tX) \cong \tQ[\NewC_{g,n}],\]
the action of $\tpsi_{\ast}^{-1}$ on $\HH_1(\tX,\partial \tX;\Q)$ is
identified with the action of $\phi$ on $\tQ[\NewC_{g,n}]$ (here we are using
the fact that $\psi\colon (X,p_n) \rightarrow (X,p_n)$ represents $\phi^{-1}$).
The lemma follows.
\end{proof}

\section{Proof II: via the Steinberg module}
\label{section:proofii}

This section uses Corollary \ref{maincorollary:identifysteinberg} to give our
second proof of Theorem \ref{maintheorem:highdimprime}.  There are two sections:
the actual proof is in \S \ref{section:steinproof}, and \S \ref{section:coinvariants}
comments on a possible way to improve our bounds.

\subsection{High-dimensional cohomology and the Steinberg module}
\label{section:steinproof}

The heart of our proof of Theorem \ref{maintheorem:highdimprime} via
Corollary \ref{maincorollary:identifysteinberg} is the following lemma.  For $1 \leq k \leq n-1$, let $\NewC_{g,n}^k$ be
the set of isotopy classes of simple closed curves on $\Sigma_{g,n}$ that bound a twice-punctured
disc whose punctures are the $k^{\text{th}}$ and $n^{\text{th}}$ ones, so $\NewC_{g,n} = \sqcup_{k=1}^{n-1} \NewC_{g,n}^k$.

\begin{lemma}
\label{lemma:orbits}
Fix some $g \geq 1$ and $n \geq 2$.  Then for all $1 \leq k \leq n-1$ and $\ell \geq 2$, the action
of $\PMod_{g,n}[\ell]$ on $\NewC_{g,n}^k$ has $\ell^{2g}$ orbits.
\end{lemma}
\begin{proof}
Just like in the proof of Theorem \ref{maintheorem:curvecpx} in \S \ref{section:identifycurve}, there
is a bijection between $\NewC_{g,n}^k$ and the set of isotopy classes of embedded arcs in $\Sigma_{g,n}$ connecting
the $k^{\text{th}}$ puncture to the $n^{\text{th}}$ puncture.  This bijection takes such an arc $\alpha$ to the
simple closed curve $\gamma$ forming the boundary of a regular neighborhood of $\alpha$.  We can therefore
regard $\NewC_{g,n}^k$ as being the set of such arcs.

Let $P = \{p_1,\ldots,p_n\}$ be a set of $n$ distinct points on $\Sigma_{g}$, which we identify with the punctures of
$\Sigma_{g,n}$.  By Lemma \ref{lemma:identifyh1}, the group $\PMod_{g,n}[\ell]$ is the kernel of the action
of $\PMod_{g,n}$ on $\HH_1(\Sigma_g,P;\Z/\ell)$.  Define a set map
\[\phi\colon \NewC_{g,n}^k \rightarrow \HH_1(\Sigma_g,P;\Z/\ell)\]
by letting $\phi(\alpha)$ be the homology class of the arc $\alpha \in \NewC_{g,n}^k$.  We have a short exact sequence
\[0 \longrightarrow \HH_1(\Sigma_g;\Z/\ell) \longrightarrow \HH_1(\Sigma_g,P;\Z/\ell) \stackrel{\partial}{\longrightarrow} \RH_0(P; \Z/\ell) \longrightarrow 0,\]
and the image of $\phi$ lies in the set
\[K = \Set{$x \in \HH_1(\Sigma_g,P;\Z/\ell)$}{$\partial(x) = p_n-p_k$}.\]
The set $K$ has $\ell^{2g}$ elements; indeed, $\HH_1(\Sigma_g;\Z/\ell)$ acts freely and transitively on it.  The map
$\phi$ is $\PMod_{g,n}[\ell]$-invariant, so it induces a map
\[\Phi\colon \NewC_{g,n}^k / \PMod_{g,n}[\ell] \rightarrow K.\]
To prove the lemma, it is enough to prove that $\Phi$ is a bijection.

Let $\PP_{g,n} \cong \pi_1(\Sigma_{g,n})$ be the point-pushing subgroup of $\PMod_{g,n}$ (c.f.\ \S \ref{section:birman}). 
The action of $\PP_{g,n}$ on $\HH_1(\Sigma_g,P;\Z/\ell)$ preserves $K$, and as we observed
in the proof of Theorem \ref{theorem:birmanlevel} its action on $K$ is given by the following formula:
\begin{equation}
\label{eqn:actionformula}
\gamma \cdot x = x + [\gamma] \quad \quad (\gamma \in \PP_{g,n}, x \in K).
\end{equation}
Here $[\gamma] \in \HH_1(\Sigma_g;\Z/\ell) \subset \HH_1(\Sigma_g,P;\Z/\ell)$ is the mod-$\ell$ homology
class of $\gamma$.

Fix $\alpha_0 \in \NewC_{g,n}^k$, and set $\kappa_0 = \phi(\alpha_0)$.  
The equation \eqref{eqn:actionformula} implies that $\PP_{g,n}$ acts transitively on
$K$, so
\[K = \Set{$\gamma \cdot \kappa_0$}{$\gamma \in \PP_{g,n}$} = \Set{$\phi(\gamma \cdot \alpha_0)$}{$\gamma \in \PP_{g,n}$}.\]
We deduce that $\phi$ is surjective and hence that $\Phi$ is surjective.

To see that $\Phi$ is injective, consider arbitrary elements $\alpha_1,\alpha_2 \in \NewC_{g,n}^k$ such that
$\phi(\alpha_1) = \phi(\alpha_2)$.  We must prove that there exists some $f \in \PMod_{g,n}[\ell]$ such that
$f \cdot \alpha_2 = \alpha_1$.  The point-pushing subgroup $\PP_{g,n}$ acts transitively on
$\NewC_{g,n}^k$, so we can write $\alpha_1 = \gamma_1 \cdot \alpha_0$ and $\alpha_2 = \gamma_2 \cdot \alpha_0$
for some $\gamma_1,\gamma_2 \in \PP_{g,n}$.  By \eqref{eqn:actionformula}, we have
\[\phi(\alpha_1) = \kappa_0 + [\gamma_1] \quad \text{and} \quad \phi(\alpha_2) = \kappa_0 + [\gamma_2],\]
so $[\gamma_1] = [\gamma_2]$.  This implies that $[\gamma_1 \gamma_2^{-1}] = 0$, so by
Theorem \ref{theorem:birmanlevel} we have 
\[f:=\gamma_1 \gamma_2^{-1} \in \PP_{g,n}[\ell] = \PP_{g,n} \cap \PMod_{g,n}[\ell].\]
This $f$ satisfies $f \cdot \alpha_2 = \alpha_1$, as desired.
\end{proof}

\begin{proof}[Proof of Theorem \ref{maintheorem:highdimprime}]
We start by recalling the statement.  Fix some $g \geq 1$ and $n \geq 2$.
Let $\nu$ be the virtual cohomological dimension of $\PMod_{g,n}$.  Finally,
fix some $\ell \geq 2$.  We must prove that the 
dimension of $\HH^{\nu}(\PMod_{g,n}[\ell];\Q)$ is at least
\begin{equation}
\label{eqn:neededsurjection}
\left(\left(n-1\right) \ell^{2g}-1\right) \cdot \dim_{\Q} \HH^{\nu-1}(\PMod_{g,n-1}[\ell];\Q).
\end{equation}
Applying Bieri--Eckmann duality, we see that
\[\HH^{\nu}(\PMod_{g,n}[\ell];\Q) \cong \HH_0(\PMod_{g,n}[\ell];\St(\Sigma_{g,n})) = (\St(\Sigma_{g,n}))_{\PMod_{g,n}[\ell]},\]
where the subscript indicates that we are taking coinvariants.  So what we must
do is produce a $\PMod_{g,n}[\ell]$-invariant epimorphism from $\St(\Sigma_{g,n})$ to
a vector space whose dimension is the quantity \eqref{eqn:neededsurjection}.

Corollary \ref{maincorollary:identifysteinberg} says that
\[\St(\Sigma_{g,n}) \cong \tQ[\NewC_{g,n}] \otimes \St(\Sigma_{g,n-1}).\]
Here the action of $\PMod_{g,n}$ on $\tQ[\NewC(\Sigma_{g,n})]$ is induced by the permutation
action of $\PMod_{g,n}$ on $\NewC_{g,n}$ and the action of $\PMod_{g,n}$ on
$\St(\Sigma_{g,n-1})$ is the one that factors through the projection
$\PMod_{g,n} \rightarrow \PMod_{g,n-1}$ that deletes the $n^{\text{th}}$ puncture.

By Harer's computation \eqref{eqn:vcdformula}
of the virtual cohomological dimension of the mapping class group, the virtual
cohomological dimension of $\PMod_{g,n-1}$ is $\nu-1$.  Let
\[\pi\colon \St(\Sigma_{g,n-1}) \rightarrow (\St(\Sigma_{g,n-1}))_{\PMod_{g,n-1}[\ell]} \cong \HH^{\nu-1}(\PMod_{g,n-1}[\ell];\Q)\]
be the evident projection.  

We have $\NewC_{g,n} = \sqcup_{k=1}^{n-1} \NewC_{g,n}^k$.  Lemma \ref{lemma:orbits} says
that $\NewC_{g,n}^k / \PMod_{g,n}[\ell]$ has $\ell^{2g}$ elements, so $\NewC_{g,n}/\PMod_{g,n}[\ell]$
has $(n-1)\ell^{2g}$ elements.  It follows that $\tQ[\NewC_{g,n}/\PMod_{g,n}[\ell]]$ is
$((n-1)\ell^{2g}-1)$-dimensional.  Let
\[\rho\colon \tQ[\NewC_{g,n}] \rightarrow \tQ[\NewC_{g,n}/\PMod_{g,n}[\ell]]\]
be the projection.

Combining $\rho$ and $\pi$, we obtain a map
\[\rho \otimes \pi\colon \St(\Sigma_{g,n}) \cong \tQ[\NewC_{g,n}] \otimes \St(\Sigma_{g,n-1}) \rightarrow \tQ[\NewC_{g,n}/\PMod_{g,n}[\ell]] \otimes \HH^{\nu-1}(\PMod_{g,n-1}[\ell];\Q).\]
This map is clearly $\PMod_{g,n}[\ell]$-invariant, and its image
has dimension equal to the quantity \eqref{eqn:neededsurjection}, as desired.
\end{proof}

\subsection{Some remarks on coinvariants}
\label{section:coinvariants}

In the proof of Theorem \ref{maintheorem:highdimprime} in \S \ref{section:steinproof}, we made
use of the map
\[\tQ[\NewC_{g,n}] \longrightarrow \tQ[\NewC_{g,n}/\PMod_{g,n}[\ell]].\]
The target of this map is a quotient of the coinvariants $\tQ[\NewC_{g,n}]_{\PMod_{g,n}[\ell]}$, and
one might think that a stronger result could be proven by using these coinvariants instead, which
a priori might be bigger.  However, the following lemma shows that they are not, at least for $g \geq 3$.
We do not know whether a better result for $g=2$ could be obtained using the coinvariants.

\begin{lemma}
\label{lemma:coinvariants}
Fix $g \geq 3$ and $n \geq 2$.  For all $\ell \geq 2$, we have
\[\tQ[\NewC_{g,n}]_{\PMod_{g,n}[\ell]} \cong \tQ[\NewC_{g,n}/\PMod_{g,n}[\ell]].\]
\end{lemma}
\begin{proof}
Recall that
\[\tQ[\NewC_{g,n}]_{\PMod_{g,n}[\ell]} = \HH_0(\PMod_{g,n}[\ell];\tQ[\NewC_{g,n}]).\]
The long exact sequence in homology associated to the short exact sequence
\[0 \longrightarrow \tQ[\NewC_{g,n}] \longrightarrow \Q[\NewC_{g,n}] \longrightarrow \Q \longrightarrow 0\]
of $\PMod_{g,n}[\ell]$-modules contains the segment
\[\HH_1(\PMod_{g,n}[\ell];\Q) \rightarrow \HH_0(\PMod_{g,n}[\ell];\tQ[\NewC_{g,n}]) \rightarrow \HH_0(\PMod_{g,n}[\ell];\Q[\NewC_{g,n}]) \rightarrow \Q \rightarrow 0.\]
Since $\PMod_{g,n}[\ell]$ contains all Dehn twists about separating curves, a result of the third author
\cite[Corollary C]{PutmanFiniteIndexNote} implies that
$\HH_1(\PMod_{g,n}[\ell];\Q) = 0$; we note that the hypothesis of the result used here include the requirement that $g \geq 3$.  We conclude that
\begin{align*}
\tQ[\NewC_{g,n}]_{\PMod_{g,n}[\ell]} &\cong
\ker(\HH_0(\PMod_{g,n}[\ell];\Q[\NewC_{g,n}]) \rightarrow \Q)\\
&= \ker(\Q[\NewC_{g,n}]_{\PMod_{g,n}[\ell]} \rightarrow \Q) \\
&= \ker(\Q[\NewC_{g,n}/\PMod_{g,n}[\ell]] \rightarrow \Q) \\
&= \tQ[\NewC_{g,n}/\PMod_{g,n}[\ell]]. \qedhere
\end{align*}
\end{proof}

\begin{remark}
The fact that $\HH_1(\PMod_{g,n}[\ell];\Q) = 0$ for $g \geq 3$ is a special case of a conjecture of
Ivanov \cite{IvanovConjecture} saying that $\HH_1(\Gamma;\Q)=0$ for all
finite-index subgroup $\Gamma$ of $\PMod_{g,n}$ with $g \geq 3$.  This conjecture
has been verified in many cases -- in addition to the paper
\cite{PutmanFiniteIndexNote} we cited above, other recent results on it includes work of Ershov-He
\cite{ErshovHe} and the third author with Wieland \cite{PutmanWieland}.  We also remark that it is important that
we are working over $\Q$ since the abelianization of $\PMod_{g,n}[\ell]$ 
contains a large amount of exotic torsion (see \cite{PutmanPicard, SatoTorsion}).
\end{remark}

\section{Applications to algebraic geometry}
\label{section:algebraicgeometry}

In this final section, we prove Theorems \ref{maintheorem:cohcd1} and
\ref{maintheorem:cohcd2}.  The actual proofs are in \S \ref{section:agproofs}, which
is preceded by the preliminary \S \ref{section:cohcd} surveying basic properties
of coherent cohomological dimension.  As we said in the introduction, to make
this material accessible to topologists we give more details than would be
necessary for an audience of algebraic geometers.

\subsection{Coherent cohomological dimension}
\label{section:cohcd}

In this section, all varieties will be defined over $\C$.

Recall that for a variety $X$, the coherent cohomological dimension of $X$, denoted
$\cohcd(X)$, is the maximal $k$ such that there exists some quasi-coherent sheaf
$\cF$ on $X$ such that $\HH^k(X;\cF) \neq 0$.  This is always finite; indeed, we have the following:

\begin{lemma}
\label{lemma:cdbound}
For any variety $X$, we have $\cohcd(X) \leq \dim(X)$.
\end{lemma}
\begin{proof}
Grothendieck's vanishing theorem \cite[Theorem 2.7]{HartshorneBook} says that for {\em any}
sheaf $\cF$ of abelian groups on $X$, we have $\HH^k(X;\cF) = 0$ for $k>\dim(X)$.
\end{proof}

The most familiar quasi-coherent sheaves are the finite-rank locally free ones, that is, sections of finite-rank algebraic
vector bundles on $X$.  The following lemma shows that in most cases it is enough to consider
only these sheaves:

\begin{lemma}
\label{lemma:locallyfree}
For a quasi-projective variety $X$, we have that $\cohcd(X)$ is the maximal $k$ such that
there exists a finite-rank locally free sheaf $\cF$ on $X$ such that $\HH^k(X;\cF) \neq 0$.
\end{lemma}
\begin{proof}
Set $k = \cohcd(X)$.  By Lemma \ref{lemma:cdbound}, we have $k<\infty$.  We must prove that
there exists some finite-rank locally free sheaf $\cF$ on $X$ such that $\HH^k(X;\cF) \neq 0$.
By definition, there exists a quasi-coherent sheaf $\cG$ on $X$ such that $\HH^k(X;\cG) \neq 0$.
Since every quasi-coherent sheaf is the direct limit of its coherent subsheaves \cite[Exercise II.5.15c]{HartshorneBook} 
and sheaf cohomology commutes with direct limits \cite[Proposition III.2.9]{HartshorneBook}, we
can assume that $\cG$ is coherent.  Since $X$ is quasi-projective, there exists some
finite-rank locally free sheaf $\cF$ along with a surjection $\pi\colon \cF \rightarrow \cG$; indeed, this
holds for projective $X$ \cite[Corollary 5.18]{HartshorneBook}, and in the quasi-projective case
we can embed $X$ as an open subvariety of a projective variety $Y$ and extend $\cG$ to a coherent sheaf
on $Y$ \cite[Exercise 5.15b]{HartshorneBook}.  Letting $\cG' = \ker(\pi)$, the long exact
sequence in cohomology associated to the short exact sequence of coherent sheaves
\[0 \longrightarrow \cG' \longrightarrow \cF \stackrel{\pi}{\longrightarrow} \cG \longrightarrow 0\]
contains the segment
\[\HH^k(X;\cF) \longrightarrow \HH^k(X;\cG) \longrightarrow \HH^{k+1}(X;\cG').\]
Since $k = \cohcd(X)$, we have $\HH^{k+1}(X;\cG') = 0$.  Since $\HH^k(X;\cG) \neq 0$, this implies
that $\HH^k(X;\cF) \neq 0$, as desired.
\end{proof}

The following says in particular that passing to a finite cover does not change the coherent
cohomological dimension:

\begin{theorem}[{Hartshorne, \cite[Proposition 1.1]{HartshorneDimension}}]
\label{theorem:finitecoh}
If $f\colon X \rightarrow X'$ is a finite surjective map between varieties, then $\cohcd(X) = \cohcd(X')$.
\end{theorem}

We now turn to some calculations.  The first is as follows.

\begin{lemma}
\label{lemma:cohpn}
We have $\cohcd(\bbP^n) = n$.
\end{lemma}
\begin{proof}
By Lemma \ref{lemma:cdbound}, we have $\cohcd(\bbP^n) \leq n$.  This inequality is an equality since
$\HH^n(\bbP^n;\cO(-n-1)) \cong \C$; see \cite[Theorem III.5.1]{HartshorneBook}.
\end{proof}

\begin{corollary}
\label{corollary:projectivedim}
If $X$ is a projective variety of dimension $n$, then $\cohcd(X) = n$.
\end{corollary}
\begin{proof}
By a sequence of linear projections from points
of projective space not lying in $X$ we can construct a finite surjective morphism $X \rightarrow \bbP^n$ (this
is a geometric version of Noether Normalization).  The corollary now follows from Theorem \ref{theorem:finitecoh}
and Lemma \ref{lemma:cohpn}.
\end{proof}

The following result shows that affine varieties behave very differently from projective ones.

\begin{theorem}[{Serre, \cite[Theorem 3.7]{HartshorneBook}}]
\label{theorem:affine}
A variety $X$ is affine if and only if $\cohcd(X) = 0$.
\end{theorem}

If $X$ is a smooth projective curve, then Corollary \ref{corollary:projectivedim} implies
that $\cohcd(X) = 1$.  Our next lemma shows that removing
any finite set of points from such an $X$ causes the coherent cohomological dimension to drop to $0$:

\begin{lemma}
\label{lemma:affinecurves}
Let $X$ be a smooth projective curve and let $P \subset X$ be a nonempty finite set.  Then
$X \setminus P$ is affine.  In other words, $\cohcd(X \setminus P) = 0$.
\end{lemma}
\begin{proof}
The Riemann--Roch theorem implies that we can find a finite map
$f\colon X \rightarrow \bbP^1$ such that $f^{-1}(\infty) = P$ (as a set, not as a
divisor -- we are ignoring multiplicities).  It follows that $f|_{X \setminus P}$ is a finite surjective
map to $\bbA^1$, which has coherent cohomological dimension $0$ by Theorem \ref{theorem:affine}.
The lemma now follows from Theorem \ref{theorem:finitecoh}.  An alternate way of concluding the
proof is to use the fact that a finite morphism is affine, so $f^{-1}(\bbA^1) = X \setminus P$
is affine.
\end{proof}

Our final two results concern the effect of morphisms on coherent cohomological dimension.  The first is as
follows:

\begin{lemma}
\label{lemma:affinemorphism}
Let $f\colon X \rightarrow Y$ be an affine morphism between varieties.  Then $\cohcd(X) \leq \cohcd(Y)$.
\end{lemma}
\begin{proof}
Let $\cF$ be a quasi-coherent sheaf on $X$.  Since $f$ is affine, the Leray spectral sequence
of $f$ degenerates (see \cite[Exercise III.8.2]{HartshorneBook}; the key point here
is Theorem \ref{theorem:affine}) to show that
\[\HH^k(X;\cF) \cong \HH^k(Y;f_{\ast} \cF) \quad \quad \text{for all $k$}.\]
The sheaf $f_{\ast} \cF$ on $Y$ is quasi-coherent, so for $k> \cohcd(Y)$ we have $\HH^k(X;\cF) = 0$.
The lemma follows.
\end{proof}

\begin{remark}
Recalling that finite morphisms are affine, one might hope that like in Theorem \ref{theorem:finitecoh},
equality always hold in Lemma \ref{lemma:affinemorphism}.  Unfortunately, this is false.
Indeed, by the ``Jouanolou trick'' (see \cite[Lemma 1.5]{JouanolouTrick}),
for any quasi-projective variety $Y$, there exists a surjective affine morphism $X \rightarrow Y$
with $X$ an affine variety (so $\cohcd(X) = 0$ by Theorem \ref{theorem:affine}, while $\cohcd(Y)$ could be
anything).
\end{remark}

Our final result concerns the coherent cohomological dimension of flat families:

\begin{lemma}
\label{lemma:flat}
Let $f\colon X \rightarrow Y$ be a flat projective morphism between quasi-projective varieties.
Letting $r$ be the dimension of the fibers of $f$, we then have $\cohcd(X) \leq \cohcd(Y)+r$.
\end{lemma}
\begin{proof}
Set $k = \cohcd(X)$.  By Lemma \ref{lemma:locallyfree}, we can find a finite-rank
locally free sheaf $\cF$ on $X$ such that $\HH^k(X;\cF) \neq 0$.  We then have the
Leray spectral sequence
\[E_2^{pq} = \HH^p(Y;R^q f_{\ast}(\cF)) \Rightarrow \HH^{p+q}(X;\cF).\]
Since the higher direct images $R^q f_{\ast}(\cF)$ are quasi-coherent, we have
$E_2^{pq} = 0$ for $p>\cohcd(Y)$.  Below we will prove that $R^q f_{\ast}(\cF) = 0$
for $q>r$, so $E_2^{pq} = 0$ for $q>r$.  Since $\HH^k(X;\cF) \neq 0$, we must
therefore have $k \leq \cohcd(Y)+r$, as desired.

Fix some $q>r$.  It remains to prove that $R^q f_{\ast}(\cF) = 0$.  Consider a point $y_0 \in Y$.  We will prove
that 
\[(R^q f_{\ast}(\cF)) \otimes k(y_0) = 0.\] 
Letting $X_{y_0} = f^{-1}(y_0)$, there is a natural map
\begin{equation*}
\label{eqn:highermap}
\eta: (R^q f_{\ast}(\cF)) \otimes k(y_0) \rightarrow \HH^q(X_{y_0};\cF_{y_0});
\end{equation*}
see \cite[\S III.12]{HartshorneBook}.
Since $X_{y_0}$ has dimension $r$, Lemma \ref{lemma:cdbound} implies that $\HH^q(X_{y_0};\cF_{y_0}) = 0$.  
It is thus enough to prove that the map $\eta$ is an isomorphism.

Since $\cF$ is locally free on $X$, it is flat over $X$ and thus flat over $Y$.  What is more, since
its target  is $0$, the map $\eta$ is trivially surjective.
Under these assumptions ($f$ a projective morphism, $\cF$ a coherent sheaf on $X$ that is flat over $Y$, surjectivity of $\eta$), we can apply
\cite[Theorem 12.11 (Cohomology and base change)]{HartshorneBook} to deduce that $\eta$ is
an isomorphism, as desired.
\end{proof}

\subsection{Algebraic geometry proofs}
\label{section:agproofs}

We finally prove Theorems \ref{maintheorem:cohcd1} and \ref{maintheorem:cohcd2}.

\begin{proof}[Proof of Theorem \ref{maintheorem:cohcd1}]
We first recall the statement.  Fix some $g \geq 2$ and $n \geq 1$.  We
must prove that $\cohcd(\Moduli_{g,n}) \geq g-1$.  Assume for the sake
of contradiction that $\cohcd(\Moduli_{g,n}) \leq g-2$.  Harer's
computation of the virtual cohomological dimension of the mapping class
group \eqref{eqn:vcdformula} says that the virtual cohomological dimension
of $\PMod_{g,n}$ is $4g-4+n$.  We will prove below that our assumption that
$\cohcd(\Moduli_{g,n}) \leq g-2$ implies that $\HH^{4g-4+n}(\PMod_{g,n}[\ell];\C) = 0$
for all $\ell \geq 3$.  However, Theorem \ref{maintheorem:highdim}
implies that $\HH^{4g-4+n}(\PMod_{g,n}[\ell];\C) \neq 0$, a contradiction.

It remains to prove that our assumption $\cohcd(\Moduli_{g,n}) \leq g-2$ implies that 
\begin{equation}
\label{eqn:toprovecohcd}
\HH^{4g-4+n}(\PMod_{g,n}[\ell];\C) = 0 \quad \quad \text{for all $\ell \geq 3$}.
\end{equation}
Fix some $\ell \geq 3$ and let $\Moduli_{g,n}[\ell]$ be\footnote{Warning:
this notation is also commonly used to denote a different space, namely, 
the moduli space of curves with a full level-$\ell$
structure, which corresponds to the kernel of the action of $\PMod_{g,n}$ on
$\HH_1(\Sigma_g;\Z/\ell)$ obtained by filling in the punctures.  The space
$\Moduli_{g,n}[\ell]$ is a cover of this.} 
the finite cover
of $\Moduli_{g,n}$ corresponding to the subgroup $\PMod_{g,n}[\ell]$ of
its (orbifold) fundamental group $\PMod_{g,n}$.  Recall that the singularities
of $\Moduli_{g,n}$ exist due to the presence of curves with automorphisms (for instance,
from the point of view of Teichm\"{u}ller theory, the singularities come from the fixed
points of the action of the mapping class group on Teichm\"{u}ller space).  Our
assumption $\ell \geq 3$ implies that the group $\PMod_{g,n}[\ell]$ is torsion-free, so
$\Moduli_{g,n}[\ell]$ is a smooth variety.

The map
$\Moduli_{g,n}[\ell] \rightarrow \Moduli_{g,n}$ is a finite surjective map,
so by Theorem \ref{theorem:finitecoh} we have
\begin{equation}
\label{eqn:cohdimass}
\cohcd(\Moduli_{g,n}[\ell]) = \cohcd(\Moduli_{g,n}) \leq g-2.
\end{equation}
Since $\Moduli_{g,n}[\ell]$ is smooth, 
we can
calculate its cohomology using de Rham cohomology.  
The Hodge--de Rham spectral sequence for $\Moduli_{g,n}[\ell]$ converges to
$\HH^{\bullet}(\Moduli_{g,n}[\ell];\C)$ and has
\[E_1^{pq} = \HH^q(\Moduli_{g,n}[\ell];\Omega^p).\]
Since the complex dimension of $\Moduli_{g,n}[\ell]$ is $3g-3+n$, we have $\Omega^q = 0$ for
$q \geq 3g-2+n$, and thus
\begin{equation}
\label{eqn:e11}
E_1^{pq} = 0 \quad \quad (q \geq 3g-2+n).
\end{equation}
Moreover, since $\Omega^q$ is a coherent sheaf on $\Moduli_{g,n}[\ell]$ equation
\eqref{eqn:cohdimass} implies that
\begin{equation}
\label{eqn:e12}
E_1^{pq} = \HH^p(\Moduli_{g,n}[\ell];\Omega^q) = 0 \quad \quad (p \geq g-1).
\end{equation}
From \eqref{eqn:e11} and \eqref{eqn:e12},
we deduce that $E_1^{pq} = 0$ whenever $p+q = 4g-4+n$.  We conclude
that \eqref{eqn:toprovecohcd} holds, as desired.
\end{proof}

Before proving Theorem~\ref{maintheorem:cohcd2} we require one further lemma.  It is a generalization to the relative setting
of the familiar fact that if $X$ is a projective variety and $D \subset X$ is an ample divisor, then
$X \setminus D$ is affine (quick proof: by the definition of an ample divisor, there exists an embedding
$X \hookrightarrow \bbP^n$ whose hyperplane section at infinity is a positive multiple of $D$, so $X \setminus D$
is a closed subvariety of $\bbA^n$).

\begin{lemma}
\label{lemma:removedivisor}
Let $f\colon X \rightarrow Y$ be a projective morphism of varieties and let $D$ be a divisor on $X$
whose intersection with each fiber $X_{y} = f^{-1}(y)$ is ample.  Then the restriction of $f$ to
$X \setminus D$ is affine.
\end{lemma}
\begin{proof}
It is enough to assume that $Y$ is affine and prove that $X \setminus D$ is affine.  
By \cite[Theorem 1.7.8]{LazersfeldI}, the divisor $D$ is ample relative to $f$.  This means
that for some $m \gg 0$, the divisor $m D$ induces an embedding $X \hookrightarrow \bbP^n_Y$ for some
$n$ such that $f$ factors as $X \hookrightarrow \bbP^n_Y \rightarrow Y$.  The hyperplane
section at infinity of $X \subset \bbP^n_Y$ is precisely $m D$.
We conclude that $X \setminus D$ is a closed subvariety of the
affine variety $\bbA^n_Y$, and hence is itself affine.
\end{proof}

\begin{proof}[Proof of Theorem~\ref{maintheorem:cohcd2}]
We first recall the statement.  Fix some $g \geq 2$ and $n \geq 1$.  We must
prove that $\cohcd(\Moduli_{g,n}) \leq \cohcd(\Moduli_g)+1$.  Set
\[\delta = \begin{cases}
1 & \text{if $n = 1$},\\
0 & \text{if $n \geq 2$}.
\end{cases}\]
We will prove that $\cohcd(\Moduli_{g,n}) \leq \cohcd(\Moduli_{g,n-1})+\delta$.

Fix some $\ell \geq 3$.  Just like in the previous proof\footnote{Unlike in the previous
proof, here we could have simply taken the moduli space of curves with a full level-$\ell$
structure; however, we have decided not to introduce yet another piece of notation.}, let 
$\Moduli_{g,n-1}[\ell]$ be the cover of $\Moduli_{g,n-1}$ corresponding to the
subgroup $\PMod_{g,n-1}[\ell]$ of its (orbifold) fundamental group $\PMod_{g,n-1}$.
Recall that $\Moduli_{g,n-1}$ is only a coarse moduli space rather than a fine
moduli space for the same reason it is not smooth: the existence of curves with
automorphisms.  Since $\ell \geq 3$, the group $\PMod_{g,n-1}[\ell]$ is torsion-free,
so $\Moduli_{g,n-1}[\ell]$ is not only smooth, but also a fine moduli space.
It thus has a universal curve 
$U_{g,n-1}[\ell] \rightarrow \Moduli_{g,n-1}[\ell]$.
Removing the $(n-1)$ distinguished points from each fiber of $U_{g,n-1}[\ell]$, we
get a space $\tModuli_{g,n}$ that fits into a commutative diagram
\[\begin{CD}
\tModuli_{g,n} @>>> \Moduli_{g,n} \\
@VV{f}V                @VVV \\
\Moduli_{g,n-1}[\ell] @>>> \Moduli_{g,n-1}.
\end{CD}\]
Both horizontal maps are finite surjective maps, so by 
Theorem \ref{theorem:finitecoh} we have
\[\cohcd(\tModuli_{g,n}) = \cohcd(\Moduli_{g,n}) \quad \text{and} \quad
\cohcd(\Moduli_{g,n-1}[\ell]) = \cohcd(\Moduli_{g,n-1}).\]
It is thus enough to prove that 
$\cohcd(\tModuli_{g,n}) \leq \cohcd(\Moduli_{g,n-1}[\ell])+\delta$.

We first deal with the case where $n=1$.  The morphism $\tModuli_{g,1} \rightarrow \Moduli_{g}[\ell]$
is a flat projective morphism whose fibers are smooth projective curves.  We can thus apply
Lemma \ref{lemma:flat} to deduce that
\[\cohcd(\tModuli_{g,1}) \leq \cohcd(\Moduli_{g}[\ell]) + 1,\]
as desired.

We now deal with the case where $n>1$.  By Lemma \ref{lemma:affinemorphism}, to prove that
$\cohcd(\tModuli_{g,n}) \leq \cohcd(\Moduli_{g,n-1}[\ell])$, it is enough to prove that
the morphism $\tModuli_{g,n} \rightarrow \Moduli_{g,n-1}[\ell]$ is affine.  Recalling
that $U_{g,n-1}[\ell]$ is the universal curve, the morphism $\pi\colon U_{g,n-1}[\ell] \rightarrow \Moduli_{g,n-1}[\ell]$
is projective, and $\tModuli_{g,n}$ is obtained from $U_{g,n-1}[\ell]$ by removing a divisor whose
intersection with each fiber $\pi^{-1}(S)$ is a sum of $(n-1)$ points.  Since all positive divisors
on an algebraic curve are ample, the result now follows from Lemma \ref{lemma:removedivisor}.
\end{proof}

\begin{footnotesize}
\noindent
\begin{tabular*}{\linewidth}[t]{@{}p{\widthof{School of Mathematics \& Statistics}+0.25in}@{}p{\widthof{Department of Mathematics}+0.25in}@{}p{\linewidth - \widthof{Department of Mathematics} - \widthof{School of Mathematics \& Statistics}- 0.5in}@{}}
{\raggedright
Tara Brendle\par
School of Mathematics \& Statistics\par
University of Glasgow\par
University Place\par
Glasgow G12 8QQ\par
UK\par
{\tt tara.brendle@glasgow.ac.uk}}
&
{\raggedright
Nathan Broaddus\par
Department of Mathematics\par
Ohio State University\par
231 W. 18th Ave.\par
Columbus, OH 43210\par
USA\par
{\tt broaddus.9@osu.edu}}
&
{\raggedright
Andrew Putman\par
Department of Mathematics\par
University of Notre Dame \par
255 Hurley Hall\par
Notre Dame, IN 46556\par
USA\par
{\tt andyp@nd.edu}}\hfill
\end{tabular*}\hfill
\end{footnotesize}

\end{document}